\documentclass[12pt,oneside]{amsart}

\usepackage{hyperref}
\usepackage{graphicx, amsmath, amssymb, amscd, amsthm, euscript, psfrag, amsfonts, enumerate, mathabx, mathrsfs, pb-diagram,pb-xy}
\usepackage{mathdots}
\usepackage[cmtip,arrow,matrix,curve,tips,frame]{xy}

\theoremstyle{plain}
\newtheorem{thm}[equation]{Theorem}

\newtheorem{rmk}[equation]{Remark}

\newtheorem{lem}[equation]{Lemma}
\newtheorem{dfn}[equation]{Definition}
\numberwithin{equation}{section}

\newcommand{\sm}{\left(\begin{smallmatrix}}
\newcommand{\esm}{\end{smallmatrix}\right)}
\newcommand{\bpm}{\begin{pmatrix}}
\newcommand{\ebpm}{\end{pmatrix}}

\newcommand{\C}{\mathbb{C}}

\newcommand{\Q}{\mathbb{Q}}
\newcommand{\R}{\mathbb{R}}
\newcommand{\Z}{\mathbb{Z}}
\newcommand{\A}{\mathbb{A}}

\newcommand{\bH}{\mathbb{H}}
\newcommand{\fund}{\mathfrak F}

\newcommand{\LtwoH}{L^2\left(SL(n, \Z)\bsl \mathbb{H}^n\right)}

\newcommand{\Ltwog}{L^2\left(SL(n, \Z)\bsl SL(n, \R)\right)}

\newcommand{\LtwoHcusp}{L^2_{\rm cusp}\left(SL(n, \Z)\bsl \bH^n\right)}
\newcommand{\LtwoHcont}{L^2_{\rm cont.}\left(SL(n, \Z)\bsl \bH^n\right)}
\newcommand {\LtwoHresi}{L^2_{\rm resi.}\left(SL(n, \Z)\bsl \bH^n\right)}

\newcommand{\bsl}{\backslash}

\newcommand{\rIm}{{\rm Im}}
\newcommand{\rRe}{{\rm Re}}

\title{Approximate Converse Theorem}
\author{Min Lee}
\address{Mathematics department, Brown university, Providence, RI}\email{minlee@math.brown.edu}

\begin{document}

\begin{abstract}
We present an approximate converse theorem which measures how close a given set of irreducible admissible unramified unitary generic local representations of $GL(n)$ is to a genuine cuspidal representation. To get a formula for the measure, we introduce a quasi-Maass form on the generalized upper half plane for a given set of local representations. We also construct an annihilating operator which enables us to write down an explicit cuspidal automorphic function.
\end{abstract}

\maketitle

\section{Introduction}

The spectral theory of non-holomorphic automorphic forms for the Poincar\'e upper half plane $\bH^2=\left\{z\in \C\;|\; \rIm(z)>0\right\}$ began with Maass in 1949. 
It is a highly non-trivial problem to show that there exist infinitely many Maass forms of the Laplacian for $L^2\left(SL(2, \Z)\bsl \bH^2\right)$. The existence of infinitely many Maass forms for $SL(2, \Z)$ was first proved by Selberg \cite{Selberg} in 1956. He introduced the trace formula as a tool to obtain Weyl's law, which gives an asymptotic count for the number of Maass forms with Laplacian eigenvalue $\left|\lambda\right| \leq X$ as $X\to\infty$. Selberg's method was extended by Miller \cite{Miller} to obtain Weyl's law for Maass forms for $SL(3, \Z)$. In 2004, M\"uller \cite{Muller} further extended Selberg's method to obtain Weyl's law for Maass forms for the congruence subgroups $\Gamma<SL(n, \Z)$, $n\geq 2$. 

Up to now, no one has found a single explicit example of a Maass form for $SL(n, \Z)$, with $n\geq 2$, although Maass \cite{Maass} discovered some examples for congruence subgroups $\Gamma< SL(2, \Z)$ of finite index by using Hecke $L$-functions. In the 1970's a number of authors considered the problem of computing Maass forms for $SL(2, \Z)$ numerically. The first notable algorithms for computing Maass forms on $SL(2, \Z)\bsl \bH^2$ are due to Stark \cite{Stark} and Hejhal \cite{Hejhal}. In 2006, Booker, Str\"ombergsson and Venkatesh \cite{Booker-Str-Venkatesh} computed the Laplace and many Hecke eigenvalues for the first few Maass forms on $PSL(2, \Z)\bsl \bH^2$ to over 1000 decimal places, based on Hejhal's algorithm. Moreover, they suggested a method of how to check the numerical computation rigorously and verified that Laplacian eigenvalues were correct up to 100 decimal places. 

Recently, Lindenstrauss and Venkatesh  \cite{Lindenstrauss-Venkatesh} obtained Weyl's law for spherical cusp forms on $G(\Z)\bsl G(\R) / K_\infty$ where $G$ is a split semisimple group over $\Q$ and $K_\infty \subset G$ is the maximal compact subgroup. In the Appendix \cite{Lindenstrauss-Venkatesh}, they explained a short constructive proof of the existence of cusp forms using Whittaker functions. This proof does not give Weyl's law, but it gives a very explicit method for constructing cuspidal functions, which was used in \cite{Booker-Str-Venkatesh}. Moreover, this constructive method can be used to attack the following ``approximate converse" problem suggested by Peter Sarnak, at the conference in 2008 on {\it Analytic number theory in higher rank groups}: 

{\it Given a positive number $X$, a set $S$ of places and a representation $\pi_v$ of $GL(n, \Q_v)$ for $v\in S$, give an algorithm to determine whether or not there is a global automorphic representation $\sigma$ whose analytic conductor is at most $X$ and whose local component at $v$ is within $\epsilon$-distance from $\pi_v$ for each place $v\in S$. 
} 

The main goal of this paper is to suggest the approximate converse theorem (Theorem \ref{t:main}) as an answer to this question for globally unramified cuspidal automorphic representations of $GL(n, \A)$ where $\A$ denotes the ring of adeles over $\Q$. 

\subsection{Quasi-Maass forms and the annihilating operator}  

Let $n\geq 2$ be an integer. For a place $v\leq \infty$, let $\pi_v$ be an irreducible admissible unramified unitary generic representation of $\Q_v^\times\bsl GL(n, \Q_v)$. Let $L(s, \pi_v)$ be the local $L$-function attached to $\pi_v$. 
Then there is the associated Satake (or Langlands) parameter $\ell_{\pi_v}=(\ell_{\pi_v, 1}, \ldots, \ell_{\pi_v, n})\in \C^n$ with $\ell_{\pi_v, 1}+\cdots +\ell_{\pi_v, n}=0$ such that 
	\begin{align}\label{e:parameters_local_representation}
	L(s, \pi_v) = \left\{\begin{array}{ll}
	\pi^{-\frac{ns}{2}} \prod\limits_{j=1}^n \Gamma\left(\frac{s+\ell_{\pi_\infty, j}}{2}\right), & \text{ if }v=\infty\\
	\prod\limits_{j=1}^n \left(1-p^{-\ell_{\pi_p, j}-s}\right)^{-1}, & \text{ if }v=p<\infty
	\end{array}\right.
	\end{align}
for $s\in \C$. 

For $v=p$ a finite prime, we have
	$$L(s, \pi_p) = \left(1- \lambda_p^{(1)}(\ell_{\pi_p}) p^{-s} +\cdots +(-1)^j \lambda_p^{(j)}(\ell_{\pi_p})p^{-js} +\cdots +(-1)^np^{-ns}\right)^{-1}$$
	where
	\begin{align}\label{e:parameter_Hecke_eigenvalue}
	\lambda_p^{(j)}(\ell_{\pi_p}) = \sum_{1\leq r_1< \cdots < r_j\leq n}  p^{-(\ell_{\pi_p, r_1}+\cdots +\ell_{\pi_p, r_j})}.
	\end{align}

For $v=\infty$ the infinity prime, we have the Whittaker function $W_J(z; \ell_{\pi_\infty}, \pm 1)$ of type $\ell_{\pi_\infty}$ in (\ref{e:Whittaker}) 
 for $z\in \bH^n\cong SL(n, \R)/SO(n, \R)$, the generalized upper half plane. 
It is a solution of differential equations given
 by Casimir operators $\mathcal C_n^{(j)}$ in (\ref{e:Casimir_op})
 with eigenvalues $\lambda_\infty^{(j)}(\ell_{\pi_\infty})\in \C$ for $j=1, \ldots, n-1$ in (\ref{e:Casimir_eigenvalue}).
 The Casimir operators generate the center of the universal enveloping algebra of the Lie algebra of $GL(n, \R)$. 
 In particular $\Delta_n=-\mathcal C_n^{(1)}$ denotes the Laplacian and 
 $$\lambda_n(\ell_{\pi_\infty}) :=-\lambda_\infty^{(1)}(\ell_{\pi_\infty})= \frac{n
 +1}{12} -\frac{1}{n(n-1)}\left(\ell_{\pi_\infty, 1}^2 +\cdots +\ell_{\pi_\infty, n}^2\right)$$
 is  the corresponding eigenvalue.
 The 
 Whittaker functions were constructed by Jacquet \cite{Jacquet}. 

The goal of this paper is to suggest a method to compare a set of local representations and a global cuspidal representation. 
In order to compare sets of local representations, we should define a ``distance" between them.
 Roughly speaking, the ``distance" can be defined by the difference between the coefficients of local $L$-functions.
 
 For $M$ a set of places of $\Q$ (including $\infty$), let 	
 	$$\Pi_M := \left\{\pi_v\;|\; v\in M\right\}$$
where each $\pi_v$ is an irreducible admissible unitary unramified generic representation of $\Q_v^\times\bsl GL(n, \Q_v)$.  
 
\begin{dfn}[{\bf Distances between sets of local representations}]\label{d:closeness}
Let $M$ and $M'$ be sets of places of $\Q$ including $\infty$ and $\Pi_M$, $\Pi'_{M'}$ as above. 
Let $S\subset M\cap M'$ be a finite subset including $\infty$. Define the distance between $\Pi_M$ and $\Pi'_{M'}$ for $S$ to be
	\begin{align}\label{e:close}
	&d_S(\Pi_M, \Pi'_{M'})\\\notag
	&\quad:= \sum_{j=1}^{n-1}\left|\lambda_\infty^{(j)}(\ell_{\pi_\infty})-\lambda_\infty^{(j)}(\ell_{\pi'_\infty})\right|^2 +\sum_{q\in S, \atop \text{ finite}} \sum_{j=1}^{\lfloor\frac{n}{2}\rfloor}\left|\lambda_q^{(j)}(\ell_{\pi_q})-\lambda_q^{(j)}(\ell_{\pi'_q})\right|^2
	\end{align}
for $\pi_v\in \Pi_M$ and $\pi'_v\in \Pi'_{M'}$. 
\end{dfn}


We use the quasi-mode construction \cite{Booker-Str-Venkatesh} for a given set of local representations $\Pi_M$ to construct an automorphic function. 

For any non-negative integers $k_1, \ldots, k_{n-1}$, let $S_{k_1, \ldots, k_{n-1}}(x_1, \ldots, x_n)$ 
be a Schur polynomial as in (\ref{e:Schur}). 
For $\pi_p\in \Pi_M$ for a prime $p$, 
define $A_{\Pi_M}(p^{k_1}, \ldots, p^{k_{n-1}}) := S_{k_1, \ldots, k_{n-1}}(p^{-\ell_{\pi_p, 1}}, \ldots, p^{-\ell_{\pi_p, n}})$. 
For any positive integers $m_1, \ldots, m_{n-1}$, we construct $A_{\Pi_M}(m_1, \ldots, m_{n-1})$ satisfying 
	$$A_{\Pi_M}(m_1, \ldots, m_{n-1})\cdot A_{\Pi_{M_1}}(m'_1, \ldots, m'_{n-1}) = A_{\Pi_M}(m_1m'_1, \ldots, m_{n-1}m'_{n-1})$$
	if $m_1\cdots m_{n-1}$ and $m_1'\cdots m_{n-1}'$ are relative prime to each other.
	 Set $A_{\Pi_M}(m_1, \ldots, m_{n-1})=0$ if there exists a prime $q\notin M$ such that $q\mid m_1\cdots m_{n-1}$. 

By combining $W_J(*; \ell_{\pi_\infty}, \pm 1)$, the Whittaker function of type $\ell_{\pi_\infty}$, and complex numbers $A_{\Pi_M}(m_1, \ldots, m_{n-1})$ for $m_1, \ldots, m_{n-1}\in \Z$,  
we construct a function for $z\in \bH^n$, which is essentially a Whittaker-Fourier expansion. 

\begin{dfn}[{\bf Quasi-Maass form}]\label{d:Quasi-Maass}
Let $M$ be a set of places over $\Q$ including $\infty$.
For $z\in \bH^n$, define 
	\begin{align}\label{e:Quasi-Maass}
	F_{\Pi_M}(z) &= \sum_{\gamma\in N(n-1, \Z)\bsl SL(n-1, \Z)}\sum_{m_1=1}^\infty\cdots\sum_{m_{n-2}=1}^\infty\sum_{m_{n-1}\neq 0} \frac{A_{\Pi_M}(m_1, \ldots, m_{n-1})}{\prod_{k=1}^{n-1}\left|m_k\right|^{\frac{k(n-k)}{2}}}\\\notag
	&\times W_J\left(\bpm m_1\cdots m_{n-2}|m_{n-1}| & & & \\ & \ddots & & \\ & & m_1 & \\ & & & 1\ebpm \cdot\bpm \gamma & \\ & 1\ebpm \; z; \ell_{\pi_\infty}, \frac{m_{n-1}}{|m_{n-1}|}\right)\;,
	\end{align}
Then $F_{\Pi_M}$ is called a quasi-Maass form of $\Pi_M$.
\end{dfn}

A quasi-Maass form $F_{\Pi_M}$ is a function on $\bH^n$ which lies in the restricted tensor product of local representations $\pi_v\in \Pi_M$.
 By definition, we can easily observe that 
 $F_{\Pi_M}$ is an eigenfunction of the Casimir operators $\mathcal C_n^{(j)}$
 with eigenvalues $\lambda_\infty^{(j)}(\ell_{\pi_\infty})$ for $\pi_\infty\in \Pi_M$, for $j=1, \ldots, n-1$. It is also an 
 eigenfunction of the Hecke operators.
  In particular, for each $p\in M$, the $\lambda_p^{(j)}(\ell_{\pi_p})$ for $\pi_p\in \Pi_M$ 
 are  eigenvalues of $F_{\Pi_M}$ of Hecke operators 
 $T_p^{(j)}$ for $j=1, \ldots, \lfloor\frac{n}
 {2}\rfloor$ given in Definition \ref{d:extend_Hecke}.

A Hecke-Maass form is a smooth function in $L^2(SL(n, \Z)\bsl \bH^n)$ which is an eigenfunction of the Casimir operators and the Hecke operators simultaneously. 
Every Hecke-Maass form is a quasi-Maass form but not vice versa. 
For an arbitrary set of local representations, the quasi-Maass form usually is not automorphic for $SL(n, \Z)$. 
Fix a fundamental domain $\fund^n\cong SL(n, \Z)\bsl \bH^n$ (described in Remark \ref{r:fundamental_domain_explicit}), and define the automorphic lifting of the quasi-Maass form as follows.

\begin{dfn}[{\bf Automorphic lifting}]\label{d:auto-lifting}
Define an automorphic lifting 
	\begin{align*}
	\widetilde F_{\Pi_M}(z) := F_{\Pi_M}(\gamma z),
	\end{align*}
for any $z\in \bH^n$ and a unique $\gamma\in SL(n, \Z)$ such that $\gamma z\in \fund^n$.
\end{dfn}

Let
	\begin{align}\label{e:fundamental_domain_Parabolic}
	\widetilde\fund^n = \underset{\gamma\in SL(n-1, \Z), \atop c_1, \ldots, c_{n-1}\in \Z}{\bigcup}\bpm & & &c_1\\ & \gamma & &\vdots\\ & & & c_{n-1}\\ 0 & \ldots & 0 &1\ebpm \;\fund^n\;,
	\end{align}
then $\widetilde F_{\Pi_M}(z) = F_{\Pi_M}(z)$ for $z\in \widetilde\fund^n$. 
The function $\widetilde F_{\Pi_M}$ is automorphic for $SL(n, \Z)$ and square-integrable.
 But it is neither smooth nor cuspidal in general. To construct a cuspidal function from $\widetilde F_{\Pi_M}$, we use an operator whose image is cuspidal, defined in \cite{Lindenstrauss-Venkatesh}. 

The approach in \cite{Lindenstrauss-Venkatesh} is based on the observation that there are strong relations between the spectrum of the Eisenstein series at different places. 
From this observation, a convolution operator was constructed, which annihilates the spectrum of the Eisenstein series. 
So the image of the operator is purely cuspidal. 
This convolution operator was used to obtain Weyl's law \cite{Lindenstrauss-Venkatesh} 
and also was used to give  a short and elementary proof which shows the existence of infinitely many Maass forms. 

For $PSL(2,\Z)\bsl \bH^2$, this operator was given explicitly in \cite{Lindenstrauss-Venkatesh} and it was used in \cite{Booker-Str-Venkatesh} to check the numerical computation rigorously. 
For $n\geq 3$, although the operator was defined in a more general case, it is quite complicated and it is not easy to describe this operator explicitly. 
In this paper, we give an explicit construction of the convolution operator whose image is purely cuspidal. The construction goes as follows. 

Fix a prime $p$. For $j=1, 2$,  set $\ell_j=(\ell_{j, 1}, \ldots, \ell_{j, n})\in \C^n$, $\ell_{j, 1}+\cdots +\ell_{j, n}=0$. Let
	\begin{align}\label{e: annihilating_eigenvalue}
	&\widehat\natural_p^n(\ell_1, \ell_2)\\\notag
	&\quad := \prod_{k=1}^{\lfloor\frac{n}{2}\rfloor}\prod_{1\leq i_1<\cdots < i_k\leq n}\prod_{1\leq j_1< \cdots < j_k\leq n} \left(1-p^{-(\ell_{1, i_1}+\cdots +\ell_{1, i_k})-(\ell_{2, j_1}+\cdots +\ell_{2, j_k})}\right).
	\end{align}
In Lemma \ref{l:construction_annihilating},  we use a Paley-Wiener type theorem \cite{Helgasson:2000} for $\widehat\natural_p^n$ and define the annihilating operator $\natural_p^n$ to be a certain polynomial in convolution operators and  Hecke operators at a prime $p$. 
Then quasi-Maass forms are eigenfunctions of the operator $\natural_p^n$. 
For a given set of local representation $\Pi_M$, we have $\natural_p^n F_{\Pi_M} = \widehat\natural_p^n(\ell_{\pi_\infty}, \ell_{\pi_p})\cdot F_{\Pi_M}$ for $\pi_\infty, \pi_p\in \Pi_M$.  
	
As in \cite{Lindenstrauss-Venkatesh}, we prove the following theorem. 

\begin{thm}\label{t:image_cuspidal}
	The space of the image of the annihilating operator $\natural_p^n$ on smooth functions in $\LtwoH$ is cuspidal and infinite dimensional. So there are infinitely many Hecke-Maass forms in $\break \LtwoH$ which are not self-dual. 
\end{thm}

We apply the annihilating operator $\natural_p^n$ to $\widetilde F_{\Pi_M}$ to construct a non-trivial cuspidal function. Since $\natural_p^n$ is a polynomial in convolution operators associated to some compactly supported distributions, we need to make the function $\widetilde F_{\Pi_M}$ to be smooth. 

For any $g\in SL(n, \R)$, by Polar decomposition, there exist $\xi_1, \xi_2\in SO(n, \R)$ and $(a_1, \ldots, a_n)\in \R^n$, $a_1+\cdots +a_n=0$ such that $g = \xi_1 \sm e^{a_1} & & \\ & \ddots & \\ & & e^{a_n}\esm \xi_2$. As in \cite{Jorgenson-Lang}, we define a polar height $\sigma(g)\geq0$ to be
	\begin{align}\label{e: polar_height}
	\sigma(g):= \sqrt{a_1^2+\cdots +a_n^2}, \quad\text{ for any }g\in SL(n, \R).
	\end{align}
Then $\sigma(g^{-1}) = \sigma(g)$ and $\sigma(g_1g_2) \leq \sigma(g_1)+\sigma(g_2)$ for any $g_1, g_2\in SL(n, \R)$. For any $\delta>0$, define
	\begin{align}\label{e:Bdelta}
	B_\delta:=\left\{g\in SL(n, \R)\;|\; \sigma(g)<\delta\right\}.
	\end{align}

For $\delta>0$, let $H_\delta$ be a standard bump function with ${\rm supp}(H_\delta)\subset B_\delta$. 
Then $\widetilde F_{\Pi_M}* H_\delta(g)=\int_{SL(n, \R)}\widetilde F_{\Pi_M}(gh^{-1})H_\delta(h)\; dh$ is a smooth function. By Lemma \ref{l:lowerbd_bump_funct}, if $0<\delta\leq \ln\left(\frac{n(n+6)}{8}\left(\sum_{j=1}^n \left|\ell_{\pi_\infty, j}+\frac{n-2j+1}{2}\right|\right)^{-1}+1\right)$, then $\widetilde F_{\Pi_M}* H_\delta$ is non-trivial. So $\natural_p^n(\widetilde F_{\Pi_M}*H_\delta)$ is a cuspidal function by Theorem \ref{t:image_cuspidal}. 

The quasi-Maass form $F_{\Pi_M}$, and its automorphic and cuspidal liftings play important roles in the approximate converse theorem. 

\subsection{Approximate Converse Theorem}
In this section we present an approximate converse theorem which measures how close a given set of irreducible admissible unramified unitary generic local representations of $GL(n)$ is to a global cuspidal representation.

For a finite prime $q$ and for $1\leq j\leq \lfloor\frac{n}{2}\rfloor$, let $G_q^{(j)}$ be a finite set of matrices which generate the Hecke operator $T_q^{(j)}$ defined in Definition \ref{d:extend_Hecke} by left translations. Let $\sharp(T_q^{(j)})$ be the number of elements in $G_q^{(j)}$. For a subset $V\subset \bH^n$, let
	$$T_q^{(j)}V := \bigcup_{\gamma\in G_q^{(j)}} \left\{\gamma z\;\left|\; z\in V \right.\right\}$$
	and
	$$(T_q^{(j)})^{-1}V := \bigcup_{\gamma\in G_q^{(j)}}\left\{z\in \bH^n\;\left|\; \gamma z\in V\right.\right\}\;.$$

\begin{thm}[Approximate Converse Theorem]\label{t:main} Let $M$ be a set of places of $\Q$ including $\infty$ properly and $S\subset M$ be a finite subset including $\infty$.
Assume that 
$\widehat\natural_p^n(\ell_{\pi_\infty}, \ell_{\pi_p})\neq 0$ for some $p\in M$ and $\pi_\infty, \pi_p\in \Pi_M$. 

For $0<\delta\leq \ln\left(\frac{n(n+6)}{8}\left(\sum_{j=1}^n \left|\ell_{\pi_\infty, j}+\frac{n-2j+1}{2}\right|\right)^{-1}+1\right)$, there exists an unramified cuspidal representation $\pi$ of $\A^\times\bsl GL(n, \A)$ such that 
	\begin{align}\label{e:main}
	&d_S(\pi, \Pi_M) <\underset{z\in B_2}{\rm sup}\left|F_{\Pi_M}(z)-\widetilde F_{\Pi_M}(z)\right|^2 \\\notag
	&\quad \times \frac{4\left(p^{- \frac{n^2-1}{2(n^2+1)}} + p^{ \frac{n^2-1}{2(n^2+1)}}\right)^{n2^{n-1}} \cdot {\rm Vol}(B_1) \cdot(A_\infty+A_{S, {\rm finite}})}{\left|\widehat\natural_p^n (\ell_{\pi_\infty}, \ell_{\pi_p})\right|^2 \cdot\int\limits_{e^{4(2^{n-1}\ln p+\delta)}}^\infty \cdots\int\limits_{e^{4(2^{n-1}\ln p+\delta)}}^\infty \left|W_J(y; \ell_{\pi_\infty}, 1)\right|^2\; d^*y},
	\end{align}
where
	$$A_\infty = \sum_{j=1}^{n-1}\left(\int_{SL(n, \R)}\left|\left(\mathcal C_n^{(j)}-\lambda_\infty^{(j)}(\ell_{\pi_\infty})\right) H_\delta(g)\right|\; dg\right)^2\;,$$
	$$A_{S, {\rm finite}}=\left\{\begin{array}{ll}
	0, & \text{ if }S=\left\{\infty\right\}\;,\\
	e^{\frac{n(n+6)}{4}\delta}\sum\limits_{q\in S, \atop \text{ finite }}\sum\limits_{j=1}^{\lfloor\frac{n}{2}\rfloor}\left(\sharp T_q^{(j)}\right)^2, & \text{ otherwise,}
	\end{array}\right.$$
	$$B_1=\left\{\begin{array}{ll}
	\left(\bH^n-\widetilde \fund^n\right)\cdot B_\delta\cap \fund^n, & \text{ if }S=\infty, \\
	\left(T_{q_{\rm max}}^{\lfloor\frac{n}{2}\rfloor}\right)^{-1}\left(\bH^n-\widetilde \fund^n\right)\cap \fund^n, & \text{ otherwise,}	
	\end{array}\right.$$
and
	$$B_2 =\left\{\begin{array}{ll}
	\fund^n\cdot B_\delta-\fund^n, & \text{ if }S=\left\{\infty\right\}, \\
	T_{q_{\rm max}}^{\lfloor\frac{n}{2}\rfloor} \fund^n -\fund^n, & \text{ otherwise .}
	\end{array}\right.$$
Here $q_{\rm max}=\max \left(S-\left\{\infty\right\}\right)$.
\end{thm}

\begin{rmk}
\begin{enumerate}
	\item If the right hand side of (\ref{e:main}) is sufficiently small for sufficiently large $S$, then by Remark 8 \cite{Brumley}, the genuine cuspidal representation $\pi$ can be uniquely determined. 
	
	\item For an unramified cuspidal representation $\sigma\cong\otimes_v'\sigma_v$ of $\A^\times\bsl GL(n, \A)$, 
	define an analytic conductor $\mathcal C(\sigma):=\prod\limits_{j=1}^n \left(1+|\ell_{\sigma_\infty, j}|\right)$ as in \cite{Brumley}. 
	Fix $Q\geq 2$. By \cite{Brumley}, for any unramified cuspidal representation $\sigma\cong\otimes_v'\sigma_v$ with $\mathcal C(\sigma)\leq Q$, there exists a prime $p\ll \log Q$ such that $\left|\widehat \natural_p^n\left(\ell_{\sigma_\infty}, \ell_{\sigma_p}\right)\right|$ is sufficiently large.
 
\end{enumerate}
\end{rmk}

There are two important ingredients in the proof of Theorem \ref{t:main}: the annihilating operator $\natural_p^n$ and the spectral decomposition of $\break\LtwoH$.  
For a given set of local representations $\Pi_M$, we can construct a cuspidal function explicitly, by applying $\natural_p^n$ to the automorphic lifting of a quasi-Maass form $F_{\Pi_M}$. The procedure is described in the previous section. Since the function is cuspidal, it is generated by Hecke-Maass forms for $SL(n, \Z)$. We can get (\ref{e:main}) by applying the Casimir operators and Hecke operators and compare the eigenvalues. 

One could try to generalize this theorem by relaxing the globally unramified condition.

Another particular question attracts our attention. How close can a given positive real number get to a Laplace eigenvalue of an actual Maass form? 
The following is a sample case of Theorem \ref{t:main}, which can be regarded as an answer to this question. 

\begin{thm}\label{t:main_laplacian} 
Let $M$ be a set of places of $\Q$ including $\infty$ properly. 
Assume that $\widehat\natural_p^n(\ell_{\pi_\infty}, \ell_{\pi_p})\neq 0$ for at least one prime $p\in M$ and $\pi_\infty, \pi_p\in \Pi_M$.
Then the Laplace eigenvalue $\lambda_n(\ell_{\pi_\infty})\in \C$ of $F_{\Pi_M}$ satisfies the following: 
for any $0<\delta\leq \ln\left(\frac{n(n+6)}{8}\left(\sum_{j=1}^n \left|\ell_{\pi_\infty, j}+\frac{n-2j+1}{2}\right|\right)^{-1}+1\right)$, 
there exists an eigenvalue $\lambda$ of the Laplacian of a Maass form for $SL(n, \Z)$ such that	
\begin{align*}
	&\left|\lambda-\lambda_n(\ell_{\pi_\infty})\right|^2\\
	&\quad < \;\underset{z\in \fund_n\cdot B_\delta}{\rm sup}\left|\widetilde F_{\Pi_M}(z)-F_{\Pi_M}(z)\right|^2\\
	&\quad\quad\quad \times \frac{\left(p^{- \frac{n^2-1}{2(n^2+1)}} + p^{ \frac{n^2-1}{2(n^2+1)}}\right)^{n2^{n-1}}\cdot {\rm Vol}\left(\left((\bH^n-\widetilde\fund^n)\cdot B_\delta\right) \cap \fund^n\right)}{\left|\widehat\natural_p^n(\ell_{\pi_\infty}, \ell_{\pi_p})\right|^2 \cdot \int\limits_{e^{4(2^{n-1}\ln p+\delta)}}^\infty\cdots\int\limits_{e^{4(2^{n-1}\ln p+\delta)}}^\infty\left|W_J(y;\ell_{\pi_\infty}, 1)\right|^2\; d^*y}\\
	&\quad\quad \quad \times
	\left[
	\frac{6(1+e^{2\delta})}{\delta^4} C_\delta\cdot {\rm Vol}(B_\delta) + 2\lambda_n(\ell_{\pi_\infty})
	\right]^2
	\end{align*}
where 
	$$C_\delta = \left(\int_{B_\delta} e^{-\frac{1}{1-(\delta^{-1}\sigma(g))^2}}\; dg\right)^{-1}\;.$$
\end{thm}
We choose a ``good" bump function $H_\delta$ to prove this theorem.

Recently, Booker and his student Bian computed the Laplace and Hecke eigenvalues for Maass forms on $SL(3, \Z)\bsl \bH^3$ \cite{Bian}, \cite{Booker}. Moreover, Mezhericher presented an algorithm for evaluating a (quasi-) Maass form for $SL(3, \Z)$ in his thesis \cite{Mezhericher}. We expect that we might use the approximate converse theorem to certify Bian's computations. 

\subsection*{Acknowledge} This paper came about through the suggestion of my thesis advisor, Dorian Goldfeld, and I would like to thank to him for his invaluable advice and continual encouragement in this project. Also I am indebted to Andrew Booker, Sug Woo Shin, Andreas Str\"ombergsson and Akshay Venkatesh for helpful comments.

\section{Quasi-Maass forms and the Annihilating operator}

We review basic facts about the Hecke-Maass forms in the first two sections. The main reference is \cite{Goldfeld:2006}.

In the introduction, we define quasi-Maass forms corresponding to the given set of local representations. By applying the annihilating operator $\natural_p^n$, it is possible to write down a cuspidal function explicitly. We study about the quasi-Maass forms and the annihilating operator in  \S\ref{ss:Quasi-Maass_forms}.

\subsection{Preliminaries}
Let $n\geq 2$ be an integer. Let 
	$$A^0(n, \R) =\left\{\left.\bpm e^{a_1} & & \\ & \ddots & \\  & & e^{a_n}\ebpm\;\right|\; {a_1, \ldots, a_n\in \R, \atop a_1+\cdots +a_n=0}\right\}\subset SL(n, \R)$$
then 
	$$\mathfrak a(n)=\left\{a=(a_1, \ldots, a_n)\in \R^n\;|\; a_1+\cdots +a_n=0\right\}$$ 
is isomorphic to the Lie algebra of $A^0(n, \R)$. For each $a = (a_1, \ldots, a_n)\in \mathfrak a(n)$, define
	$$\exp a = \bpm e^{a_1} & & \\ & \ddots & \\ & & e^{a_n}\ebpm \in A^0(n, \R).$$ 

Let
	$$N(n, \R) := \left\{\left. \bpm 1 & x_{1, 2} & x_{1, 3} & \ldots & x_{1, n}\\ & 1 & x_{2, 3} & \ldots & x_{2, n}\\ & & \ddots & &\vdots\\ & & & 1 & x_{n-1, n}\\ & & & & 1\ebpm\;\right| \; x_{i, j}\in \R, \; \text{ for }1\leq i< j\leq n\right\}.$$
Define the generalized upper half plane $\bH^n$ to be a set of matrices $z=xy\in SL(n, \R)$ 
such that  
	\begin{align}\label{e:Iwasawa_x}
	x = \bpm 1 & x_{1, 2} & x_{1, 3} & \ldots & x_{1, n}\\ & 1 & x_{2, 3} & \ldots & x_{2, n}\\ & & \ddots & &\vdots\\ & & & 1 & x_{n-1, n}\\ & & & & 1\ebpm\;\in \; N(n, \R)	
	\end{align}
and 
	\begin{align}\label{e:Iwasawa_y}
	&y=a_{y_1, \ldots, y_{n-1}}\\\notag
	&:=\left(\prod_{j=1}^{n-1}y_j^{n-j}\right)^{-\frac{1}{n}}\cdot \bpm y_1\cdots y_{n-1} & & & \\ & y_1\cdots y_{n-2} & & \\ & & \ddots & \\ & & & 1\ebpm\;\in \;A^0(n, \R)
	\end{align}
for $y_1, \ldots, y_{n-1}>0$. By the Iwasawa decomposition, any $g\in SL(n, \R)$ can be written uniquely as $g=xy \xi$ with $x\in N(n, \R)$,  $y\in A_0(n, \R)$ and $\xi\in SO(n, \R)$. So the generalized upper half plane $\bH^n$ can be identified with the quotient $SL(n, \R)/SO(n, \R)$.

Define ${\rm Iw}_Y(g)\in \mathfrak a(n)$ to be
	\begin{align}\label{e:Iw_Y}
	\exp({\rm Iw}_Y(g))= a_{y_1, \ldots, y_{n-1}}.
	\end{align}

\begin{lem}\label{l:Relations_Iwasawa_Polar}
For any $g\in SL(n, \R)$, let $a_{y_1, \ldots, y_{n-1}} = \exp({\rm Iw}_Y(g))$ for $y_1, \ldots, y_{n-1} > 0$. Then we have
	$$e^{-2\sigma(g)} \; \leq \; y_1\; \leq e^{2\sigma(g)}$$
	and
	$$e^{-4\sigma(g)}\; \leq \; y_j \; \leq e^{4\sigma(g)},\hskip 10pt\text{(for $j=2, \ldots, n-1$)}.$$
\end{lem}

\begin{proof}[Proof of Lemma \ref{l:Relations_Iwasawa_Polar}]
By the Iwasawa decomposition and the polar decomposition, we have
	$$g=z\xi_{\rm Iw} = \xi_1\bpm e^{a_1} & & \\ & \ddots & \\  & & e^{a_n}\ebpm\xi_2, \quad (\xi_{\rm Iw}, \xi_1, \xi_2\in SO(n, \R)),$$
for $(a_1, \ldots, a_n)\in \mathfrak a(n)$, where $z=xa_{y_1, \ldots, y_{n-1}}\in \bH^n$ for $x\in N(n, \R)$ and $a_{y_1, \ldots, y_{n-1}} = \exp({\rm Iw}_Y(g))$.  Then
	$$z\cdot^tz = x\cdot a_{y_1, \ldots, y_{n-1}}^2\cdot^tx = k_1\cdot \bpm e^{2a_1} & & \\ & \ddots & \\ & & e^{2a_n}\ebpm \cdot ^tk_1$$
Let $y_n'=\prod_{j=1}^{n-1}y_j^{-\frac{2(n-j)}{n}}$ and $y_j'=y_n'\cdot(y_1\cdots y_{n-j})^2$ for $j=1, \ldots, n-1$. Comparing the diagonal parts, we have
	$$e^{-2\sigma(g)} \; \leq\; y_j' \;\leq\; e^{2\sigma(g)}, \hskip 20pt \text{(for $j=1, \ldots, n$)}$$
since $\sigma(g^{-1})=\sigma(g)$. Therefore we have
	$$e^{-4\sigma(g)} \leq y_j'\cdot y_n'^{-1} = (y_1\cdots y_{n-j})^2\leq e^{4\sigma(g)}$$
for $j=1, \ldots, n-1$. 
\end{proof}

Let $\mathfrak a^*(n)=\left\{\alpha=(\alpha_1, \ldots, \alpha_n)\in \R^n\;|\; \alpha_1+\cdots +\alpha_n=0\right\}$ 
be the dual space of $\mathfrak a(n)$ and $\mathfrak a^*_\C(n)=\mathfrak a^*(n)\otimes_\R\C$. For $\alpha, \alpha'\in \mathfrak a(n)$ or $\mathfrak a_\C^*(n)$, let
	$$\left<\alpha, \alpha'\right> = \sum_{j=1}^n \alpha_j\alpha'_j\;.$$
then $\left<, \right>$ is a (complex) symmetric bilinear form and it is positive definite on $\mathfrak a^*(n)$. For $\alpha\in \mathfrak a^*(n)$, we put $||\alpha||=\sqrt{\left<\alpha, \alpha\right>}$. Let $W_n$ denote the Weyl group of $(\R^\times\cap SL(n, \R))\bsl SL(n, \R)$, consisting of all $n\times n$ matrices in $SL(n, \Z)\cap SO(n, \R)$ which have exactly one $\pm 1$ in each row and column. The Weyl group $W_n$ acts on $\mathfrak a(n)$ and $\mathfrak a_\C^*(n)$ as a permutation group. 

Let $\mathfrak{gl}(n, \R)$ be the Lie algebra of $GL(n, \R)$ with the Lie bracket $[, ]$ given by $[X, Y] = XY-YX$ for $X, Y\in \mathfrak{gl}(n, \R)	$. The universal enveloping algebra of $\mathfrak{gl}(n, \C) = \mathfrak{gl}(n, \R)\otimes_\R \C$ can be realized as an algebra of differential operators $D_X$ acting on smooth functions $f: GL(n, \R)\to \C$. The action is given by
	$$D_X f(g) :=\left.\frac{\partial}{\partial t}f(g\exp(tX))\right|_{t=0}$$
for $X\in\mathfrak{gl}(n, \R)$ where $\exp(tX) = \sum_{n=0}^\infty \frac{(tX)^n}{n!}$. For $1\leq i, \; j\leq n$, let $E_{i, j}\in \mathfrak{gl}(n, \R)$ be the matrix with $1$ at the $(i, j)$th entry and $0$ elsewhere. Let $D_{i, j}=D_{E_{i, j}}$ for $1\leq i, \; j\leq n$. For $j=1, \ldots, n-1$, we define Casimir operators $\mathcal C_n^{(j)}$ given by
	\begin{align}\label{e:Casimir_op}
	\mathcal C_n^{(j)} = \frac{(n-j-1)!}{n!}\sum_{i_1=1}^n \cdots\sum_{i_j=1}^n D_{i_1, i_2}\circ D_{i_2, i_3}\circ\cdots \circ D_{i_{j+1}, i_1}
	\end{align}
where $\circ$ is the composition of differential operators. Let $\Delta_n := -\mathcal C_n^{(1)}$ be the Laplace operator. Let $\mathcal Z^n$ be the center of the universal enveloping algebra of $\mathfrak{gl}(n, \R)$. It is well known that $\mathcal Z^n \cong \R\left[\mathcal C_n^{(1)}, \ldots, \mathcal C_n^{(n-1)}\right]$.

There is a standard procedure to construct simultaneous eigenfunctions of all differential operators of $D\in\mathcal Z^n$. For $\ell = (\ell_1, \ldots, \ell_n)\in\mathfrak a^*_\C(n)$, define
	\begin{align}\label{e:eigenfunction_Casimir}
	\varphi_\ell(g) := \prod_{j=1}^{n-1}y_j^{\sum\limits_{k=1}^{n-j}\left(\ell_k+\frac{n-2k+1}{2}\right)} ,
	\end{align}
where $\exp\left({\rm Iw}_Y(g)\right)=a_{y_1, \ldots, y_{n-1}}$ for $g\in SL(n, \R)$. 
Then $\varphi_\ell$ is a simultaneous eigenfunction of $\mathcal Z^n$. For $j=1, \ldots, n-1$, define $\lambda_\infty^{(j)}(\ell)\in \C$ to be the eigenvalue of $\mathcal C_n^{(j)}$ for $\varphi_\ell$, such that
	\begin{align}\label{e:Casimir_eigenvalue}
	\mathcal C_n^{(j)}\varphi_\ell(g)=\lambda_\infty^{(j)}(\ell)\cdot\varphi_\ell(g), \hskip 20pt (\ell\in \mathfrak a_\C^*(n))\;.
	\end{align}
	
\begin{lem}\label{l:Laplacian_eigenvalue} Let $n\geq 2$ and $\ell=(\ell_1, \ldots, \ell_n)\in\mathfrak a_\C^*(n)$. The Laplace eigenvalue is
	\begin{align}\label{e:Laplacian_eigenvalue}
	\lambda_n(\ell)=-\lambda_\infty^{(1)}(\ell)=\frac{n+1}{12} -\frac{1}{n(n-1)}(\ell_1^2+\cdots+\ell_n^2)\;.
	\end{align}
\end{lem}	
\begin{proof}[Proof of Lemma \ref{l:Laplacian_eigenvalue}]
For any $y\in A^0(n, \R)$, consider $\Delta_n\varphi_\ell(y)$. Then
	\begin{multline*}
	\Delta_n\varphi_\ell(y) \\
	=-\frac{1}{n(n-1)}\left\{\sum_{j=1}^n D_{j, j}\circ D_{j, j} \varphi_\ell(y)+\sum_{1\leq i< j\leq n} \left(D_{i, j}\circ D_{j, i} +D_{j, i}\circ D_{i, j}\right) \varphi_\ell(y)\right\}.
	\end{multline*}
For any $x\in N(n, \R)$, we have $\varphi_\ell (yx) = \varphi_\ell(y)$. So $D_{j, i} \varphi_\ell(y)=0$ for $1\leq i<  j\leq n$. 
For $1\leq i <  j\leq n$, we have
	$$D_{i, j}\circ D_{j, i}+D_{j, i}\circ D_{i, j} = 2D_{i, j}\circ D_{j, i} + D_{j, j}-D_{i, i}\;.$$
Therefore 
	\begin{multline*}
	\Delta_n\varphi_\ell(y) = -\frac{1}{n(n-1)}\left\{\sum_{j=1}^n D_{j, j}^2 \varphi_\ell(y)+\sum_{1\leq i\;<\; j\leq n}(D_{j, j}-D_{i, i})\varphi_\ell(y)\right\}\\
	=-\frac{1}{n(n-1)}\left\{\sum_{j=1}^n (D_{j, j}^2-(n-2j+1)D_{j, j})\varphi_\ell(y)\right\}\\
	=-\frac{1}{n(n-1)}\sum_{j=1}^n \left\{\left(\frac{n-2j+1}{2}+\ell_j\right)^2 -(n-2j+1)\left(\frac{n-2j+1}{2}+\ell_j\right)\right\}\\
	\times\varphi_\ell(y).
	\end{multline*}
\end{proof}

For $\ell\in\mathfrak a_\C^*(n)$, we define the spherical function of type $\ell$: 
	\begin{align}\label{e:spherical_function}
	\beta_\ell(g):=\int_{SO(n, \R)}\varphi_\ell(\xi g)\; d\xi
	\end{align}
for $g\in SL(n, \R)$. Here $d\xi$ is the normalized Haar measrue $d\xi$ on $SO(n, \R)$. 
Then the spherical function $\beta_\ell$ satisfies the followings: 
	\begin{itemize}
	\item $\beta_\ell(\xi_1 g\xi_2) = \beta_\ell(g)$ for any $\xi_1, \; \xi_2\in SO(n, \R)$
	\item $\beta_\ell$ is an eigenfunction of the Casimir operators $\mathcal C_n^{(j)}$ with an eigenvalue $\lambda_\infty^{(j)}(\ell)$ for $j=1, \ldots, n-1$
	\item $\beta_\ell(1)=1$
	\end{itemize}
	
We identify $\LtwoH$ with the right $SO(n, \R)$-invariant subspace of $\Ltwog$. If 
$k$ is a bi-$SO(n, \R)$-invariant compactly supported smooth function on $SL(n, \R)$,
 it gives rise to a convolution operator $f\to f*k$ on $\LtwoH$
given by
	$$f*k(g) = \int_{SL(n, \R)}f(gh^{-1})k(h)\; dh\;.$$
More generally, one can consider convolution with compactly supported distributions instead of functions. In this case, the convolution operator is well defined only on suitable smooth functions $f$, for example on $C^\infty\left(SL(n, \Z)\bsl \bH^n\right)$. 

Let 
$k$ be a bi-$SO(n, \R)$-invariant compactly supported smooth function on $SL(n, \R)$. 
For any $\ell\in\mathfrak a_\C^*(n)$, the spherical function $\beta_\ell$ is an eigenfunction of the corresponding convolution operator. 
The spherical transform $\widehat k(\ell)\in \C$ is defined to be the corresponding eigenvalue: 
	\begin{align}\label{e:spherical_transform}
	\left(\beta_\ell * k\right)(g) = \widehat k(\ell)\cdot \beta_\ell(g)
	\end{align}
for $g\in SL(n, \R)$. 
The inverse of the spherical transform is given explicitly in terms of the Plancherel measure $\mu_{\rm Planch}$ on $\mathfrak a_\C^*(n)$ by 
	$$k(g) = \int_{i\mathfrak a^*(n)} \widehat k(\alpha)\beta_\alpha(g)\; d\mu_{\rm Planch}(\alpha)\quad \text{(\cite{Helgasson:2000})}.$$

\subsection{Automorphic functions}\label{ss:Maass_Eisenstein}

The group $SL(n, \Z)$ acts on $\bH^n$ discretely. 
Fix $a, b\geq 0$. We define the Siegel set $\Sigma_{a, b}\subset \bH^n$ to be the set of all matrices of the form
	$$\bpm 1 & x_{1, 2} & x_{1, 3} & \ldots & x_{1, n}\\ & 1 & x_{2, 3} & \ldots & x_{2, n}\\ & & \ddots & & \vdots\\ & & & 1 & x_{n-1, n}\\ & & & & 1\ebpm \cdot a_{y_1,\ldots, y_{n-1}}$$
	with $|x_{i, j}|\leq b$ for $1\leq i\;<\; j\leq n$ and $y_i >a$ for $1\leq i \leq n-1$.
Let $\Sigma_a := \Sigma_{a, \infty}$. 

The Sigel set $\Sigma_{\frac{\sqrt 3}{2}, \frac{1}{2}}$ is a good approximation of a fundamental domain: $\bH^n = \bigcup_{\gamma\in SL(n, \Z)} \gamma \Sigma_{\frac{\sqrt 3}{2}, \frac{1}{2}}$.

An automorphic function for $SL(n, \Z)$ is a function $f: \bH^n\to \C$ such that $f(\gamma z) = f(z)$ for any $\gamma\in SL(n, \Z)$ and $z\in \bH^n$. Consider $\LtwoH$ to be the space of automorphic functions $f: \bH^n\to \C$ satisfying
	$$||f||_2^2 :=\int_{SL(n, \Z)\bsl \bH^n} \left|f(z)\right|^2\; d^*z\;<\;\infty$$
where $d^*z = d^*x \; d^*y$ is the left invariant $GL(n, \R)$-measure on $\bH^n$. 
Here
	$$d^*x = \prod_{1\leq i< j\leq n}dx_{i, j},\quad \text{ and }\quad d^*y = \prod_{k=1}^{n-1}y_k^{-k(n-k)-1}dy_k.$$
For $f_1,\; f_2\in \LtwoH$, define the inner product
	$$\left<f_1, f_2\right>:=\int_{SL(n, \Z)\bsl \bH^n} f_1(z)\overline{f_2(z)}\; d^*z\;.$$
	
Let $R$ be a commutative ring with identity $1$. For positive integers $n\geq 2$ and $1\leq n_1, \ldots, n_r\leq n$ with $n_1+\cdots +n_r=n$, define
	$$P_{n_1, \ldots, n_r}(n,R) :=\left\{\left.\bpm A_1 & & * \\ & \ddots & \\ & & A_r\ebpm\in SL(n, R)\; \right|\; {A_i\in SL(n_i, R), \atop 1\leq i\leq r}\right\}$$
	to be the standard parabolic subgroup of $SL(n, R)$ associated to $(n_1, \ldots, n_r)$. The integer $r$ is termed the rank of the parabolic subgroup $P_{n_1, \ldots, n_r}(n, R)$. Define
	$$M_{n_1, \ldots, n_r}(n, R):=\left\{\left.\bpm A_1 & & \\ &\ddots & \\ & & A_r\ebpm\;\right|\; A_i\in SL(n_i, R), \; 1\leq i\leq r\right\}$$
	to be the standard Levi subgroup of $P_{n_1, \ldots, n_r}(n, R)$. Define
	$$N_{n_1, \ldots, n_r}(n, R):=\left\{\bpm I_{n_1} & & * \\ & \ddots & \\ && I_{n_r}\ebpm \in SL(n, R)\right\}$$
	where $I_k$ is the $k\times k$ identity matrix for an integer $k\geq 1$, to be the unipotent radical of $P_{n_1, \ldots, n_r}(n, R)$.

The automorphic function $f$ for $SL(n, \Z)$ is cuspidal if
		$$\int\limits_{(SL(n, \Z)\cap N_{n_1, \ldots, n_r}(n, \Z))\bsl N_{n_1, \ldots, n_r}(n, \R)} f(uz)\; d^*u=0$$
	for any partition $n_1+\cdots+n_r=n$ and $r\geq 1$. 	
Let $\LtwoHcusp$ denote the space of automorphic cuspidal functions in $\LtwoH$. 
	
For each positive integer $N\geq 1$, define
	\begin{align}\label{e:Hecke_matrix}
	G_N:=\left\{\left.\bpm c_1 & c_{1, 2} & \ldots & c_{1, n}\\& c_2 & \ldots & c_{2, n}\\ & & \ddots & \vdots \\ & & & c_n\ebpm \;\right|\;{c_1\cdots c_n=N,\; c_1, \ldots, c_n\;>\;0\atop 0\leq c_{i, j} < c_i \hskip 5pt (1\leq i< j\leq n)}\right\}\;.
	\end{align}
Let $f: \bH^n\to \C$ be a function. For each integer $N\geq 1$, we define a Hecke operator 
	\begin{align}\label{e:Hecke}
	T_Nf(z) :=\frac{1}{N^{\frac{n-1}{2}}}\sum_{\gamma\in G_N} f\left(\gamma z\right).
	\end{align}
Clearly $T_1$ is the identity operator. If $f\in \LtwoH$ then $T_Nf\in \LtwoH$. For $n=2$, the Hecke operators are self-adjoint with respect to the inner product. For $n\geq 3$, the Hecke operator is no longer self-adjoint, but the adjoint operator is again a Hecke operator and the Hecke operator commutes with its adjoint, so it is a normal operator. 

If a smooth function $f\in L_{\rm cusp}(SL(n, \Z)\bsl \bH^n)$ is a simultaneous eigenfunction of Casimir operators $\mathcal C_n^{(j)}$ for $j=1, \ldots, n-1$ and Hecke operators $T_N$ for any $N\geq 1$, then $f$ is called a Hecke-Maass form.

If $f$ is a Hecke-Maass form, then there exists $\ell_\infty(f)\in \mathfrak a_\C^*(n)$ such that
	$$\mathcal C_n^{(j)} f = \lambda_\infty^{(j)}(\ell_\infty(f))\cdot f.$$

For $\ell\in \mathfrak a_\C^*(n)$ and $\epsilon=\pm 1$, define
	\begin{align}\label{e:Whittaker}
	W_J(z; \ell, \epsilon):=\int_{N(n, \R)} \varphi_\ell\left(\bpm & & & (-1)^{\lfloor\frac{n}{2}\rfloor} \\ & & 1 & \\ &  \iddots  & & \\ 1 & &  &\ebpm \cdot \bpm 1 & u_{1, 2} & \ldots & u_{1, n} \\ & 1 & \ldots & u_{2, n}\\ & & \ddots & \vdots\\ & & & 1\ebpm \; z\right)\\\notag
	\times e^{-2\pi i(\epsilon \cdot u_{1, 2}-u_{2, 3}-\cdots -u_{n-1, n})}\; d^*u. 
	\end{align}
to be Jacquet's Whittaker function of type $\ell$. Then 
	$$\mathcal C_n^{(j)}W_J(z;\ell, \epsilon)=\lambda_\infty^{(j)}(\ell)\cdot W_J(z; \ell, \epsilon)$$ 
	for $j=1, \ldots, n-1$. For any $u= \sm 1 & u_{1, 2} & \ldots & u_{1, n}\\ & 1& \ldots & u_{2, n}\\ & & \ddots & \vdots\\ & & &1\esm\in N(n, \R)$, we have
	$$W_J(u z; \ell, \epsilon ) = e^{2\pi i(-\epsilon \cdot u_{1, 2}+u_{2, 3}+\cdots +u_{n-1, n})}\cdot W_J(z;\ell, \epsilon)$$
	for any $z\in \bH^n$. Moreover, 
	$$\int\limits_{\Sigma_{\frac{\sqrt 3}{2}, \frac{1}{2}}}  \left|W(z; \ell, \epsilon)\right|^2 d^*z< \infty.$$

By (9.1.2) \cite{Goldfeld:2006}, every Hecke-Maass form $f$ has a Fourier-Whittaker expansion of the form
	\begin{align}\label{e:Maass_Fourier}
	f(z) = &\sum_{\gamma\in N_{n-1}\bsl SL(n-1, \Z)}\sum_{m_1=1}^\infty \cdots\sum_{m_{n-2}=1}^\infty\sum_{m_{n-1}\neq 0} \frac{A_f(m_1, \ldots, m_{n-1})}{\prod_{k=1}^{n-1}|m_k|^{\frac{k(n-k)}{2}}}\\\notag
	&\hskip 5pt \times W_J\left(\bpm m_1\cdots|m_{n-1}| & & & \\ & \ddots & &\\ & & m_1 & \\ & & & 1\ebpm \cdot\bpm \gamma & \\ & 1\ebpm\; z; \ell_{\infty}(f), \frac{m_{n-1}}{|m_{n-1}|}\right)
	\end{align}
where $A_f(m_1, \ldots, m_{n-1})\in \C$. We assume that $A_f(1, \ldots, 1)=1$. Then $T_N f = A_f(N, 1, \ldots, 1)\cdot f$ for any integer $N\geq 1$ and $A_f(m_1, \ldots, m_{n-1})$ satisfies the following (multiplicative) relation \cite{Goldfeld:2006}:
 for $(m_1, \ldots, m_{n-1})\in \Z^{n-1}$, and an integer $m\geq 1$, we have
	\begin{align}\label{e:Hecke_relation}
	&A_f(m, 1, \ldots, 1) A_f(m_1, \ldots, m_{n-1})\\\notag
	&\hskip 20pt=\sum_{\prod_{j=1}^n c_j=m, \atop c_1|m_1, \ldots, c_{n-1}|m_{n-1}} A_f\left(\frac{m_1c_n}{c_1}, \ldots, \frac{m_j c_{j-1}}{c_j}, \ldots, \frac{m_{n-1}c_{n-2}}{c_{n-1}}\right)
	\end{align}
	and here $c_1, \ldots, c_n>0$. 
Since Hecke operators are normal, we have
	$$A_f(m_{n-1}, \ldots, m_1)=\overline{A_f(m_1, \ldots, m_{n-1})}.$$

For any non-negative integers $k_1, \ldots, k_{n-1}$, let
	\begin{align}\label{e:Schur}
	S_{k_1, \ldots, k_{n-1}}(x_1, \ldots, x_n) := \frac{\left| \bpm x^{k_1+\cdots +k_{n-1}+n-1} & \ldots & x_n^{k_1+\cdots +k_{n-1}+n-1}\\ x_1^{k_1+\cdots +k_{n-2}+n-2} & \ldots & x_n^{k_1+\cdots +k_{n-2}+n-2}\\ \vdots & \vdots & \vdots \\ 
x_1^{k_1+1} & \ldots & x_n^{k_1+1}\\ 1 & \ldots & 1\ebpm \right|}
{\left|\bpm x_1^{n-1} & \ldots & x_n^{n-1}\\ x_1^{n-2} & \ldots & x_n^{n-2}\\ \vdots & \vdots & \vdots \\ x_1 & \ldots & x_n\\ 1 & \ldots & 1\ebpm\right|}
	\end{align}
be a Schur polynomial.

Since $f$ is an eigenfunction of Hecke operators, there exist $\ell_p(f) = (\ell_{p, 1}(f), \ldots, \ell_{p, n}(f))\in \mathfrak a_\C^*(n)$ for any finite prime $p$, such that 
	\begin{align}\label{e:Schur-Hecke}
	A_{f}(p^{k_1}, \ldots, p^{k_{n-1}}) = S_{k_1, \ldots, k_{n-1}}(p^{-\ell_{p, 1}(f)}, \ldots, p^{-\ell_{p, n}(f)}).
	\end{align}

\begin{dfn}\label{d:extend_Hecke}
Let $n\geq 2$ be an integer and fix a prime $p$. For $j=1, \ldots, n-1$, define
	$$T_p^{(j)} :=\sum_{k=0}^{j-1} (-1)^k T_{p^{k+1}}T_p^{(j-k-1)}$$
where $T_{p^r}^{(1)} = T_{p^r}$ for any integer $r\geq 0$ and $T_p^{(0)}$ is an identity operator. 
\end{dfn}

By the multiplicative relations (\ref{e:Hecke_relation}), we have
	\begin{align*}
	T_p^{(j)}f= A_f(\underbrace{1, \ldots, 1, p}_{j}, 1, \ldots, 1)\cdot f, \quad (\text{ for }j=1, \ldots, n-1)
	\end{align*}
for any prime $p$. Then
	$$\lambda_p^{(j)}(\ell_p(f)) = A_f(\underbrace{1, \ldots, 1, p}, 1, \ldots, 1)$$
for $j=1, \ldots, n-1$.
	
Let $n\geq 2$ be an integer. If a Maass form $f$ satisfies 
	$$f(z)= \widetilde f(z) := f\left(w\cdot\;^t(z^{-1})\;\cdot w\right), \quad w=\bpm  & & &(-1)^{\lfloor\frac{n}{2}\rfloor} \\  & & 1 & \\  & \iddots & & \\  1& & &\ebpm$$
then $f$ is called a self-dual Maass form.

\subsection{Quasi-Maass forms and the annihilating operator}\label{ss:Quasi-Maass_forms}

Let $M$ be a set of places over $\Q$ including $\infty$ and $\Pi_M$ be a set of local representations as given in the introduction. 
We construct a quasi-Maass form $F_{\Pi_M}$ of $\Pi_M$ on $\bH^n$, 
which lies in the restricted tensor product of local representations $\pi_v\in \Pi_M$ in Definition \ref{d:Quasi-Maass}. 
Then for $j=1, \ldots, n-1$, we have $\mathcal C_n^{(j)}F_{\Pi_M} = \lambda_\infty^{(j)}(\ell_{\pi_\infty})\cdot F_{\Pi_M}$ and $T_q^{(j)}F_{\Pi_M} = \lambda_q^{(j)}(\ell_{\pi_q})\cdot F_{\Pi_M}$. 
	
To define the automorphic lifting of quasi-Maass forms, we fix a fundamental domain for $SL(n, \Z)\bsl \bH^n$. We define $\fund^n$ to be the susbset of the Siegel set $\Sigma_{\frac{\sqrt 3}{2}, \; \frac{1}{2}}$, which contains $\Sigma_{1, \frac{1}{2}}$, satisfying:
	\begin{itemize}
	\item for any $z\in \bH^n$, there exists $\gamma\in SL(n, \Z)$ such that $\gamma z\in \fund^n$ ;
	\item for any $z\in \fund^n$, $\gamma z\notin\fund^n$ for any $I_n\neq \gamma\in SL(n, \Z)$ where $I_n$ is the $n\times n$ identity matrix. 
	\end{itemize}
Then $\fund^n$ becomes a fundamental domain for $SL(n, \Z)$.

In Definition \ref{d:auto-lifting}, we defined the automorphic lifting $\widetilde F_{\Pi_M}$ of a quasi-Maass form $F_{\Pi_M}$ with respect to the fixed fundamental domain $\fund^n$ and $\widetilde F_{\Pi_M}\in \LtwoH$. 

We get a smooth automorphic lifting of a quasi-Maass form $F_{\Pi_M}$ defined via $\widetilde F_{\Pi_M} * \kappa$ for any nonzero smooth compactly supported bi-$SO(n, \R)$-invariant function $\kappa$ on $SL(n, \R)$. 

For $\delta>0$,  if a smooth compactly supported bi-$SO(n, \R)$-invariant function $H_\delta$
 satisfies the following conditions:
	\begin{itemize}
	\item ${\rm supp}(H)\subset B_\delta$ ;
	\item $H_\delta(g) = H_\delta(g^{-1})$ for any $g\in SL(n, \R)$ ;
	\item $H_\delta(g)\; \geq \; 0$ ;
	\item $\int_{SL(n, \R)} H_\delta (g)\; dg = 1$
	\end{itemize}
then it is called the standard bump function.
By Lemma \ref{l:Relations_Iwasawa_Polar}, for $z\in \Sigma_{e^{4\delta}}$, we have
	$$\widetilde F_{\Pi_M}* H_\delta(z) = F_{\Pi_M}* H_\delta(z) = \widehat H_\delta(\ell_{\pi_\infty})\cdot F_{\Pi_M}(z).$$ 
 By the following lemma, it is always possible to choose  
$\delta>0$ such that $\widetilde F_{\Pi_M}* H_\delta$ is non-trivial.

\begin{lem}\label{l:lowerbd_bump_funct}
For any $\ell=(\ell_1, \ldots, \ell_n)\in \mathfrak a_\C^*(n)$ satisfying $\left|\rRe(\ell_j)\right|< \frac{1}{2}$, let
	\begin{align}\label{e:lowerbd_bump_funct}
{\rm LB}_{\delta} (\ell):=
	 1-	 \frac{4\left(e^{\frac{n(n+6)}{4}\cdot \delta}-1\right)}{ n(n+6)}\sum\limits_{j=1}^{n}\left|\ell_j+\frac{n-2j+1}{2}\right|
	\end{align}
for $\delta>0$, 
where $\lambda_n(\ell)$ is the 
eigenvalue of the Laplacian $\Delta_n$ for $\varphi_\ell$ as in Lemma \ref{l:Laplacian_eigenvalue}. 

Choose $\delta>0$ such that $LB_\delta(\ell)>0$. Then
	$$\left|\widehat H_\delta(\ell)\right| \; >\; {\rm LB}_{\delta}(\ell).$$

\end{lem}

\begin{proof}[Proof of Lemma \ref{l:lowerbd_bump_funct}]
For $\ell\in\mathfrak a_\C^*(n)$, we have
	\begin{align*}
	\left|\widehat H_\delta(\ell) -1\right|
	&\leq \int_{SL(n, \R)} H_\delta(g)\; \left|\varphi_\ell(g)-1\right|\; dg \; \leq \;\underset{g\in B_\delta}{\rm sup}\left|\varphi_\ell(g)-1\right|.
	\end{align*}
So, 
	$$\left|\widehat H_\delta(\ell)\right| \; \geq \; 1-\underset{g\in B_\delta}{\rm sup} \left|\varphi_\ell(g)-1\right|\;.$$
For $g\in SL(n,\R)$ we have $\exp({\rm Iw}_Y(g)) = (y_1', \ldots, y_n')\in \mathfrak a(n)$ 
and $\varphi_\ell(g) = e^{\left(\sum_{j=1}^n \left(\ell_j+\frac{n-2j+1}{2}\right)\cdot \ln y_j'\right)}$.
For any $a, b\in \R$, we have $e^{a+ib}-1 = (a+ib)\int_1^e x^{a+ib-1}\; dx$. So
	\begin{align*}
	\left|\varphi_\ell(g)-1\right|&\leq \left|\sum_{j=1}^n \left(\ell_j+\frac{n-2j+1}{2}\right)\cdot\ln y_j'\right|\cdot\frac{e^{\sum_{j=1}^n \left({\rm Re}(\ell_j)+\frac{n-2j+1}{2}\right)\cdot \ln y_j'}-1}{\sum_{j=1}^n \left({\rm Re}(\ell_j)+\frac{n-2j+1}{2}\right)\cdot\ln y_j'}
	\end{align*}
if $ \sum_{j=1}^{n}\left(\rRe(\ell_j)+\frac{n-2j+1}{2}\right)\cdot\ln y_j'\neq 0$. Otherwise,
	\begin{align*}
	\left|\varphi_\ell(g)-1\right|&\leq \left|\sum_{j=1}^n \left(\ell_j+\frac{n-2j+1}{2}\right)\cdot\ln y_j'\right|.	
	\end{align*}
By Lemma \ref{l:Relations_Iwasawa_Polar}, for $g\in B_\delta$, we have
	$$-\delta\;\leq \; \ln y_j'\;\leq \; \delta, \hskip 20pt \text{(for any $j=1, \ldots, n$)}\;.$$
So, 
	\begin{align*}
	\left|\sum_{j=1}^n \left(\ell_j+\frac{n-2j+1}{2}\right)\cdot \ln y_j'\right| \;\leq \; \delta\cdot \sum_{j=1}^n \left|\ell_j + \frac{n-2j+1}{2}\right|
	\end{align*}
by Cauchy-Schwartz inequality. 
Since $\left|\rRe(\ell_j)\right|<\frac{1}{2}$, we have 
	\begin{multline*}
	\left|\sum_{j=1}^n\left(\rRe(\ell_j)+\frac{n-2j+1}{2}\right)\ln y_j'\right| 
	\leq \delta\cdot \sum_{j=1}^n \left|\rRe(\ell_j)+\frac{n-2j+1}{2}\right| \\
	\leq \delta\cdot \frac{n}{2}\left(\frac{n}{2}+3\right).
	\end{multline*}	
Since $\underset{t\to0}{\lim}\frac{e^t-1}{t}=1$ and $\frac{e^t-1}{t}$ is increasing, we have
	\begin{align*}
	\left|\varphi_\ell(g)-1\right| 
	&< \frac{4\left(e^{\frac{n(n+6)}{4}\cdot \delta}-1\right)}{n(n+6)}\left(\sum_{j=1}^{n}\left|\ell_j+\frac{n-2j+1}{2}\right|^2\right)^\frac{1}{2}.
	\end{align*}
\end{proof}

Take 
	$$0<\delta\leq \ln\left(\frac{n(n+6)}{8}\left(\sum_{j=1}^n \left|\ell_{\pi_\infty, j}+\frac{n-2j+1}{2}\right|\right)^{-1}+1\right).$$
Since $\pi_\infty$ is irreducible admissible unramified unitary generic representation of $\R^\times\bsl GL(n, \R)$, we have $\left|\rRe(\ell_{\pi_\infty, j})\right|< \frac{1}{2}$ for $j=1, \ldots, n$ \cite{GH:2011}. So $\left|\widehat H_\delta(\ell_{\pi_\infty})\right| >\frac{1}{2}$ and $\widetilde F_{\Pi_M}* H_\delta$ is a non-trivial, smooth automorphic function on $SL(n, \Z)\bsl \bH^n$.

In (\ref{e: annihilating_eigenvalue}), we defined $\widehat\natural_p^n(\ell_1, \ell_2)$ for $\ell_1, \ell_2\in \mathfrak a_\C^*(n)$. In the following lemma, we apply a Paley-Wiener type theorem to $\widehat\natural_p^n$ and construct the annihilating operator $\natural_p^n$ explicitly as a polynomial in convolution operators and in Hecke operators. 

\begin{lem}
\label{l:construction_annihilating}
	Let $n\geq 2$ be an integer and fix a prime $p$. There exists an operator denoted $\natural_p^n$, which is a polynomial in convolution operators (associated to some compactly supported bi-$SO(n, \R)$-distributions) and in Hecke operators at $p$, satisfying
	$$\natural_p^n f(z) = \widehat\natural_p^n\left(\ell_\infty(f), \ell_p(f)\right)\cdot f(z), \hskip 20pt (z\in \bH^n)\;.$$
Here $f$ is a smooth function on $\bH^n$ which is an eigenfunction of $\mathcal Z^n$ of type $\ell_\infty(f)$ and also an eigenfunction of Hecke operators at $p$, with eigenvalues as in (\ref{e:Schur-Hecke}) for $\ell_p(f)$. 
\end{lem}
\begin{rmk}
Before proving Lemma \ref{l:construction_annihilating}, we give an example of $\natural_p^n$ for the cases $n=2$ and $n=3$. 
\begin{enumerate}[{\bf (i)}]
\item For $n=2$, we have 
	\begin{align*}
		\natural_p^2 = T_{p^2}+T_p^2-2T_p\mathcal L_\kappa +1
	\end{align*}
where $\mathcal L_\kappa$ is the convolution operator associated to the distribution $\kappa$ such that $\widehat{\kappa}(\ell)=p^{\ell_1}+p^{\ell_2}$ for any $\ell=(\ell_1, \ell_2)\in\mathfrak a_\C^*(2)$. This operator satisfies $\natural_p^2 = \aleph\circ\aleph$ for the operator $\aleph$ constructed in \S2, \cite{Lindenstrauss-Venkatesh}.\\

\item Let $n=3$. For $j=1, 2, 3$, define the compactly supported bi-$SO(3, \R)$-distributions $\kappa_{\pm j}$ such that
	\begin{align*}
	&\widehat{\kappa_1}(\ell)=p^{\ell_1}+p^{\ell_2}+p^{\ell_3}, \hskip 20pt
	\widehat{\kappa_{-1}}(\ell) = p^{-\ell_1}+p^{-\ell_2}+p^{-\ell_3},\\\notag
	&\widehat{\kappa_2}(\ell) = -\widehat{\kappa_{-1}}(\ell)^2+3\widehat{\kappa_1}(\ell), \hskip 20pt \widehat{\kappa_{-2}}(\ell) = \widehat{\kappa_1}(\ell)^2-3\widehat{\kappa_{-1}}(\ell), \\\notag
	&\widehat{\kappa_3(\ell)} = -\widehat{\kappa_2}(\ell)\cdot\widehat{\kappa_1}(\ell), \hskip 10pt\text{ and }\hskip 10pt \widehat{\kappa_{-3}}(\ell) = -\widehat{\kappa_{-2}}(\ell)\cdot\widehat{\kappa_{-1}}(\ell),
	\end{align*}
	for any $\ell=(\ell_1, \ell_2, \ell_3)\in \mathfrak a_\C^*(3)$. Then
	\begin{align*}	
	\natural_p^3 &= T_p\mathcal L_{\kappa_3} + T_p^2\mathcal L_{\kappa_2}-T_p^3-T_p(T_p^{(2)})^2\mathcal L_{\kappa_1} \\\notag
	&\hskip 50pt+T_p^2T_p^{(2)}\mathcal L_{\kappa_{-1}}+(T_p^{(2)})^2\mathcal L_{\kappa_{-2}}+(T_p^{(2)})^3 + T_p^{(2)}\mathcal L_{\kappa_{-3}}.
	\end{align*}
\end{enumerate}
\end{rmk}
\begin{proof}[Proof of Lemma \ref{l:construction_annihilating}]
Let $W_n$ be the Weyl group of $SL(n, \R)$. For any $w_1, \; w_2\in W_n$, we have
	$$\widehat\natural_p^n(w_1.\ell_1, w_2.\ell_2) = \widehat\natural_p^n(\ell_1, \ell_2),$$
and $\widehat\natural_p^n(\ell_1, \ell_2)$ is holomorphic in $\ell_1, \ell_2\in \mathfrak a_\C^*(n)$. 

For $1\leq k \leq \lfloor\frac{n}{2}\rfloor$ and  $1\leq r\leq d_k(n)=\frac{n!}{k!(n-k)!}$, consider  homogeneous degree $k\cdot r$ symmetric polynomials $B_{r, k}$ in $n$ variables, defined by
	\begin{align*}
	& \prod_{1\leq j_1 <\cdots < j_k\leq n} \left(1-xp^{-(\alpha_{j_1}+\cdots +\alpha_{j_k})}\right)\\
	&=1-B_{1, k}(\alpha)x +\cdots + (-1)^r B_{r, k}(\alpha)x^r  +\cdots +(-1)^{d_k(n)}x^{d_k(n)}
	\end{align*}
for any $\alpha=(\alpha_1, \ldots, \alpha_n)\in\mathfrak a_\C^*(n)$. By \cite{Goldfeld:2006}, Hecke eigenvalues can be described in Schur polynomials in $n$ variables, and they are a linear basis for the space of homogeneous symmetric polynomials in $n$ variables. So $B_{r, k}(\alpha)$ is an eigenvalue of a linear combination of  Hecke operators at $p$.

By using an analogous of the Paley-Wiener theorem \cite{Helgasson:2000} for distributions, we show that there exist compactly supported bi-$SO(n, \R)$-invariant distributions whose spherical transform is $B_{r, k}(\alpha)$. 

For $1\leq k \leq \lfloor\frac{n}{2}\rfloor$, define homogeneous symmetric polynomials $a_{j, k}$ and $b_{j, k}$ by
	\begin{multline*}
	\prod_{1\leq j_1< \cdots < j_k\leq n}\prod_{1\leq i_1 <\cdots < i_k\leq n} \left(1-p^{-(\ell_{1, i_1}+\cdots +\ell_{1, i_k})-(\ell_{2, j_1}+\cdots +\ell_{2, j_k})}\right)\\
	=\prod_{1\leq i_1< \cdots < i_k \leq n} \left(\sum_{r=0}^{d_k(n)}B_{r, k}(\ell_1)p^{-r(\ell_{2, i_1}+\cdots +\ell_{2, i_k})}\right)
	=\sum_{j=0}^{\widetilde d_k(n)} a_{j, k}(\ell_1) \cdot b_{j, k}(\ell_2)
	\end{multline*}
for $\ell_1 = (\ell_{1, 1}, \ldots, \ell_{1, n}), \; \ell_2 = (\ell_{2, 1}, \ldots, \ell_{2, n})\in\mathfrak a_\C^*(n)$ and some positive integer $\widetilde d_k(n)$. Then $a_{j, k}(\ell_1)$ is a polynomial in $B_{r, k}(\ell_1)$. By symmetry, $b_{j, k}(\ell_2)$ is also a polynomial in $B_{r, k}(\ell_2)$. So, there exist compactly supported bi-$SO(n, \R)$-invariant distributions $\kappa_j^{(k)}$ whose spherical transform is $a_{j, k}(\ell_1)$. For each $1\leq k \leq \lfloor \frac{n}{2}\rfloor$ and $1\leq j\leq \widetilde d_k(n)$, let $\mathcal L_{\kappa_j^{(k)}}$ be the convolution operator associated to the distribution $\kappa_j^{(k)}$. 

Moreover, there exist Hecke operators $S_j^{(k)}$ such that 
	$$S_j^{(k)} f = b_{j, k}(\ell_p(f))\cdot f$$
where $f$ is an eigenfunction of Hecke operators with parameter $\ell_p(f)\in\mathfrak a_\C^*(n)$. 

Therefore  	
	\begin{align}\label{e:annihilate_polynomial}
	\natural_p^n = \prod_{k=1}^{\lfloor\frac{n}{2}\rfloor} \left(\sum_{j=0}^{\widetilde d_k(n)} S_j^{(k)} \mathcal L_{\kappa_j^{(k)}}\right)
	\end{align}
	and 
	$$\natural_p^n f = \widehat\natural_p^n (\ell_\infty(f), \ell_p(f)) \cdot f$$
where $f$ is an eigenfunction of Casimir operators and the Hecke operators. 	
\end{proof}

Since we use distributions to define the annihilating operator $\natural_p^n$, 
the operator is well defined in the space of smooth functions. 
For $\delta>0$, let $H_\delta$ be a standard bump function and we define the operator $\natural_p^n H_\delta$ to be
	$$\natural_p^n H_\delta f = \natural_p^n(f*H_\delta)$$
for a function $f: \bH^n\to \C$ which makes the integral convergent. 
Then
	$$\widehat{\natural_p^n H_\delta}(\ell_1, \ell_2) = \widehat\natural_p^n(\ell_1, \ell_2) \cdot\widehat H_\delta(\ell_1), \quad\text{ (for any $\ell_1, \ell_2\in \mathfrak a_\C^*(n)$)}.$$
By the Paley-Wiener Theorem \cite{Helgasson:2000} and Lemma \ref{l:construction_annihilating},
 the operator $\natural_p^n H_\delta$ is a polynomial in convolution operators (associated to bi-$SO(n, \R)$-invariant, compactly supported smooth functions),
  and in Hecke operators at the prime $p$.

\begin{lem}\label{l:annihilate_image_nonzero}
Let $M$ be a set of places over $\Q$ including $\infty$. Let $F_{\Pi_M}$ be a quasi-Maass form for $\Pi_M$ and $\widetilde F_{\Pi_M}$ be an automorphic lifting of $F_{\Pi_M}$.
Assume that $T>4( 2^{n-1}\ln p+\delta)$ for a given $\delta>0$. Then for any $z\in \Sigma_{e^T}$, 
	$$\natural_p^n H_\delta \widetilde F_{\Pi_M}(z) = \widehat\natural_p^n(\ell_{\pi_\infty}, \ell_{\pi_p})\cdot\widehat H_\delta(\ell_{\pi_\infty})\cdot F_{\Pi_M}(z)\;.$$
\end{lem}
\begin{proof}[Proof of Lemma \ref{l:annihilate_image_nonzero}] Let $\kappa$ be a compactly supported function with support in $B_b=\left\{z\in \bH^n\;|\; \sigma(z)\leq b\right\}$ for some $b>0$. For $t> e^{4b}$, for any $z\in \Sigma_t$, assume that $\sigma(zh^{-1})\leq b$ for some $h\in SL(n, \R)$. Let $\exp\left({\rm Iw}_Y(z)\right) = a_{y_1, \ldots, y_{n-1}}$ and $\exp\left({\rm Iw}_Y(h)\right) = a_{v_1, \ldots, v_{n-1}}$. 
Then, by Lemma \ref{l:Relations_Iwasawa_Polar}, for $j=1, \ldots, n-1$, we have
	$$e^{-4b}\leq\frac{y_j}{v_j}\leq e^{4b}.$$
So, 
	$$v_j  \geq y_j\cdot e^{-4b} \geq t\cdot e^{-4b} >  1.$$
Then $\widetilde F_{\Pi_M}(h) = F_{\Pi_M}(h)$ because $\Sigma_1 \subset \widetilde\fund^n$. So for $z\in \Sigma_t$, we have
	$$
	\widetilde F_{\Pi_M}* \kappa(z) 
	= \int_{z\cdot B_b} \widetilde F_{\Pi_M}(h)\kappa(zh^{-1})\; dh
	= \widehat\kappa(\alpha_\infty)\cdot F_{\Pi_M}(z).
	$$
	
For integers $c_1, \ldots, c_n\geq 1$ and $c_{i, j}\geq 0$ for $1\leq i < j\leq n$, and $z\in \bH^n$, we have
	\begin{align}\label{e:matrix_left}
	\bpm c_1 & c_{1, 2} & \ldots & c_{1, n} \\ & c_2 & \ldots & c_{2, n}\\ & & \ddots & \vdots\\ & & & c_n\ebpm z = x'\cdot a_{\frac{c_{n-1}}{c_n}y_1, \ldots, \frac{c_{n-j}}{c_n}y_j, \ldots, \frac{c_1}{c_n}y_{n-1}}
	\end{align}
	for $x'\in N(n, \R)$ and $\exp({\rm Iw}_Y(z)) = a_{y_1, \ldots, y_{n-1}}$. 
	So, if $\frac{y_j}{c_n}\geq 1$ for $1\leq j\leq n-1$, then $\sm c_1 & c_{1, 2} & \ldots & c_{1, n} \\ & c_2 & \ldots & c_{2, n}\\ & & \ddots & \vdots\\ & & & c_n\esm z\in \Sigma_1$. 
	
As in (\ref{e:annihilate_polynomial}), $\natural_p^n$ is a polynomial in Hecke operators and convolution operators associated with compactly supported distributions $\kappa_j^{(k)}$ for $k=1, \ldots, \lfloor\frac{n}{2}\rfloor$ and $j=0, \ldots, \frac{n!}{k!(n-k)!}$. Moreover, the spherical transform $\left|\widehat\kappa_j^{(k)}(\ell)\right| \ll e^{\frac{n!}{k!(n-k)!} \ln p \|{\rm Re}(\ell)\|}$ for any $\ell\in \mathfrak a_\C^*(n)$ and the implied constant depends on $n$ and $k$. 
So, the operator $\natural_p^n H_\delta$ is a polynomial in Hecke operators and convolution operators associated with compactly supported functions which have support in $B_{b}$ for $b\leq 2^{n-1}\ln p +\delta$, by Paley-Wiener's theorem \cite{Helgasson:2000}. 

Since Hecke operators are generated by left translations as in (\ref{e:matrix_left}), and the largest possible $c_n$ is $p^{2^{n-1}}$ for $\natural_p^n$. Therefore, for any $z\in\Sigma_{e^T}$ for $T> 4(2^{n-1}\ln p +\delta)$, we obtain
	$$\natural_p^n H_\delta\widetilde F_{\Pi_M}(z) = \widehat \natural_p^n(\ell_{\pi_\infty}, \ell_{\pi_p})\cdot\widehat H_\delta(\ell_{\pi_\infty})\cdot F_{\Pi_M}(z)\;.$$
\end{proof}

So if $\widehat\natural_p^n(\ell_{\pi_\infty}, \ell_{\pi_p})\neq 0$, then  $\natural_p^n\left(\widetilde F_{\Pi_M}* H_\delta\right)(z)$ is non trivial for $\delta$ satisfying ${\rm LB}_\delta(\ell_{\pi_\infty})>0$. By Theorem \ref{t:image_cuspidal}, we show that $\natural_p^n\left(\widetilde F_{\Pi_M}* H_\delta\right)(z)$ is a smooth cuspidal automorphic function on $SL(n, \Z)\bsl \bH^n$. We prove Theorem \ref{t:image_cuspidal} in the next section.

\begin{rmk}\label{r:fundamental_domain_explicit}
By \S2 in \cite{Grenier}, we get the following explicit description of the fundamental domain $\fund^n$. Let $n\geq 2$ be an integer and $\overline{\fund^n}$ be the closure of the fundamental domain $\fund^n$. 
\begin{enumerate}[(1)]
	\item For $n=2$, the closure of the fundamental domain $\overline{\fund^n}$ is the set of $z=\sm 1 & x\\ 0 & 1\esm \sm y^{\frac{1}{2}} & 0 \\ 0 & y^{-\frac{1}{2}}\esm\in \bH^2$ for $x, y\in \R$ and $y>0$ satisfying
		$$x^2+y^2\geq 1\hskip 10pt \text{ and }\hskip 10pt |x|\leq \frac{1}{2}\;.$$
	\item For $n >2$, the closure of the fundamental domain $\overline{\fund^n}$ is the set of
		$$z=\bpm & &  &x_1 \\ &I_{n-1} &&\vdots\\ & & &x_{n-1}\\ 0 & \ldots & 0 & 1\ebpm y_1^{-\frac{n-1}{n}}\bpm & & & 0 \\ & y_1z' & & \vdots\\ & & & 0 \\ 0 & \ldots & 0 & 1\ebpm$$
		for $x_1, \ldots, x_{n-1}\in \R$ and $y_1>0$ satisfying the following conditions:
		\begin{enumerate}[(i)]
		\item $z'\in \overline{\fund^{n-1}}$ ;
		\item for any $\bpm & & & b_1\\ & * & & \vdots\\ & & & b_{n-1} \\ c_1 & \ldots & c_{n-1} & a\ebpm\in SL(n, \Z)$, we have
		$$(a+c_1x_1+\cdots +c_{n-1}x_{n-1})^2+y_1^2 \left(\prod_{j=2}^{n-1} y_j^{n-j}\right)^{\frac{2}{n-1}}\bpm c_1 & \ldots & c_{n-1}\ebpm z' \;^tz' \bpm c_1\\ \vdots \\ c_{n-1}\ebpm\geq 1\;;$$
		where $\exp\left({\rm Iw}_Y(z')\right) = \left(\prod_{j=2}^{n-1}y_j^{n-j}\right)^{-\frac{1}{n-1}}\cdot \sm y_2\cdots y_{n-1} & & & \\ & y_2\cdots y_{n-2} & & \\ & & \ddots & \\ & & & 1\esm$.
		\item $|x_j|\leq \frac{1}{2}$ for $j=1, \ldots, n-1$. 
		\end{enumerate}
\end{enumerate}
\end{rmk}

\subsection{Proof of Theorem \ref{t:image_cuspidal}}

As suggested in the Appendix \cite{Lindenstrauss-Venkatesh}, we prove Theorem \ref{t:image_cuspidal} for the annihilating operator $\natural_p^n$. 

\begin{rmk}
Let $\ell_1=(\ell_{1, 1}, \ldots, \ell_{1, n}), \; \ell_2=(\ell_{2, 1}, \ldots, \ell_{2, n})\in \mathfrak a_\C^*(n)$.
By definition, we have  
	$$\widehat\natural_p^n(\ell_1, \ell_2) = 0$$
 whenever $\left(\ell_{1, i_1}+\cdots+\ell_{1, i_r}\right)+ \left(\ell_{2, j_1}+\cdots +\ell_{2, j_r}\right)=0$ for any $1\leq r\leq n$.
By \cite{Lindenstrauss-Venkatesh}, it can be proved that the image of the annihilating operator $\natural_p^n$ is cuspidal. Here we give an explicit proof. 
\end{rmk}

For $\delta>0$, let $H_\delta$ be a standard bump function. 
   Then the operator $\natural_p^n H_\delta$ can be defined for the functions in $L^2\left(\bH^n\right)$ and $\widehat{\natural_p^n H_\delta}(\ell_1, \ell_2) = \widehat\natural_p^n(\ell_1, \ell_2)\cdot\widehat H_\delta(\ell_1)$ for any $\ell_1, \; \ell_2\in \mathfrak a_\C^*(n)$.
	
The Langlands spectral decomposition states that
	\begin{multline*}
	\LtwoH \\
	 = \LtwoHcont \oplus \LtwoHresi \oplus \LtwoHcusp
	\end{multline*}
where $L^2_{\rm cusp}$ denote the space of Maass forms, $L^2_{\rm resi.}$ consists of iterated residues of Eisenstein series and $L^2_{\rm cont.}$ is the space spanned by integrals of Eisenstein series.
 For any $f\in \LtwoH$, there exists $f_{\rm cont.}\in L^2_{\rm cont.}$, $f_{\rm resi.}\in L^2_{\rm resi.}$ and $f_{\rm cusp}\in L^2_{\rm cusp}$ such that
	$$f(z) = f_{\rm cont.}(z) +f_{\rm resi.}(z)+f_{\rm cusp}(z) \;.$$

Our goal is to show that $\natural_p^nH_\delta f_{\rm cont.} = \natural_p^n H_\delta f_{\rm resi.} \equiv 0$, therefore $\natural_p^n H_\delta f(z) = \natural_p^n H_\delta f_{\rm cusp}(z)$. We should show that for any Eisenstein series $E$, $\natural_p^n E = 0$ for any prime $p$.

Review some facts on Eisensteins series on $SL(n, \Z)\bsl \bH^n$ \cite{Goldfeld:2006}. 
Let $n\geq 2$ be an integer. 
For each partition $n=n_1+\cdots +n_r$ with rank $1\leq r\leq n$, we have the factorization
	$$P_{n_1, \ldots, n_r}(n, \R) = N_{n_1, \ldots, n_r}(n, \R)\cdot M_{n_1, \ldots, n_r}(n, \R)\;.$$
It follows that for any $g\in P_{n_1, \ldots, n_r}(n, \R)$, we have
	$$g\in N_{n_1, \ldots, n_r}(n, \R) \cdot\bpm \mathfrak m_{n_1}(g) & 0 & \ldots & 0 \\ & \mathfrak m_{n_2}(g) & \ldots & 0 \\ & & \ddots & \vdots \\ & & & \mathfrak m_{n_r}(g)\ebpm, $$
where  $\mathfrak m_{n_i}(g)\in SL(n_i, \R)$ for $i=1, \ldots, r$. 
	
Let $n\geq 2$ be an integer and fix a partition $n=n_1+\cdots +n_r$ with $1\leq n_1,\ldots, n_r\leq n$. 
For each $i=1, \ldots r$, let $\phi_i$ be either a Maass form for $SL(n_i, \Z)\bsl \bH^{n_i}$ of type $\ell_\infty(\phi_i)\in \mathfrak a_\C^*(n_i)$ or a constant with $\ell_\infty(\phi_i) =(0, \ldots, 0)$. 
For $t=(t_1, \ldots, t_r)\in \C^r$ with $n_1t_1+\cdots +n_rt_r=0$, define a function
	$$\varphi_{n_1, \ldots, n_r}(\cdot\; ; t; \phi_1, \ldots, \phi_r) : P_{n_1, \ldots, n_r}(n, \R)\to \C$$
	by the formula
	$$\varphi_{n_1, \ldots, n_r}(g; t; \phi_1, \ldots, \phi_r):=\prod_{i=1}^r \phi_i(\mathfrak m_{n_i}(g))\cdot\left|\det(\mathfrak m_{n_i}(g))\right|^{t_i}$$
for $g\in P_{n_1, \ldots, n_r}(n, \R)$. We can check that $\varphi_{n_1, \ldots, n_r}(g; t;\phi_1, \ldots, \phi_r) = \varphi_{n_1, \ldots, n_r}(z; t;\phi_1, \ldots, \phi_r)$ for $g=z\xi$ with $z\in \bH^n$ and $\xi\in SO(n, \R)$. 
Define the Eisenstein series by the infinite series
	\begin{align}\label{e:EisensteinSeries}
	E(z)&=E_{n_1, \ldots, n_r}(z; t; \phi_1,\ldots, \phi_r) \\\notag
	&:=\sum_{\gamma\in P_{n_1, \ldots, n_r}(n, \Z)\cap SL(n, \Z)\bsl SL(n, \Z)} \varphi_{n_1, \ldots, n_r}(\gamma z; t; \phi_1, \ldots, \phi_r)
	\end{align}
for $z\in \bH^n$. Then the Eisenstein series $E$ is an eigenfunction of $\mathcal Z^n$ of type $\ell_\infty(E)$. The Eisenstein series is also an eigenfunction of Hecke operators with a parameter $\ell_p(E)$ for any prime $p$, if $\phi_1, \ldots, \phi_r$ are Hecke eigenfunctions. The Eisenstein series are not contained in $\LtwoH$, but they generate the continuous and residual spectrum in $\LtwoH$. 

The following lemma shows that an Eisenstein series is controlled by few parameters for the archimedean. 
\begin{lem}\label{l:Eisenstein_type_all} Let $n\geq 2$ be an integer. Fix a partition $n=n_1+\cdots +n_r$ with $1\leq n_1, \ldots, n_r<n$. For each $i=1, \ldots, r$, let $\phi_i$ be either a Hecke-Maass form for $SL(n_i, \Z)\bsl \bH^{n_i}$ of type $\ell_\infty(\phi_i)\in\mathfrak a_\C^*(n_i)$ or a constant with $\ell_\infty(\phi_i) = (0, \ldots, 0)$. Let $t=(t_1, \ldots, n_r)\in \C^r$ with $n_1t_1+\cdots +n_rt_r=0$. Let $E(z):=E_{n_1, \ldots, n_r}(z; t;\phi_1, \ldots, \phi_r)$ be an Eisenstein series (\ref{e:EisensteinSeries}). Let $\eta_1=0$ and $\eta_i=n_1+\cdots +n_{i-1}$ for $i=2, \ldots, r$. Then for $i=1, \ldots, r$ and $\eta_i+1\leq j \leq \eta_i+n_i$, we have
	\begin{align}\label{e:Eisenstein_type_all}
	\ell_{v, j}(E) = (-1)^\delta \left(\frac{n_i-n}{2}+t_i+\eta_i\right)+\ell_{v, j-\eta_i}(\phi_i)
	\end{align}
where $\delta=\left\{\begin{array}{ll} 0, & \text{ if } v=\infty; \\ 1, & \text{ if } v<\infty\;. \end{array}\right.$
\end{lem}
\begin{proof}[Proof of Lemma \ref{l:Eisenstein_type_all}]

By Proposition 10.9.1 \cite{Goldfeld:2006}, the Eisenstein series $E(z)$ is an eigenfunction of Casimir operators of type $\ell_\infty(E)$. 

For an integer $N\geq 1$, let $A_{\phi_i}(N)\in \C$ be the Hecke eigenvalue of $T_{N}$ for $\phi_i$ for $i=1, \ldots, r$. Then by Proposition 10.9.3 \cite{Goldfeld:2006}, the Eisenstein series $E(z)$ is an eigenfunction of the Hecke operators $T_{p^k}$ (for any $k\geq 0$ and prime $p$) with eigenvalues
	$$A_E(p^k) = p^{-\frac{k(n-1)}{2}}\sum_{k_1+\cdots +k_r=k, \atop 0\leq k_j\in \Z} \prod_{j=1}^r \left(A_{\phi_j}(p^{k_j})\cdot p^{k_j\left(\frac{n_j-1}{2}+t_j+\eta_j\right)}\right)\;\;.$$
By using the multiplicative relations (\ref{e:Hecke_relation}), we get the formula (\ref{e:Eisenstein_type_all}). 
\end{proof}

By the lemma above, for any Eisenstein series $E$, we have
	$$\widehat\natural_p^n(\ell_\infty(E), \ell_p(E)) = 0$$
for any prime $p$. 
So $\natural_p^n E=0$ for any prime $p$. Moreover, for any constant $C\in \C$, we have $\natural_p^n C=0$ for any prime $p$. 
Since the invariant integral operators and Hecke operators preserve the space of cuspidal functions, we have $\natural_p^n H_\delta f = \natural_p^n H_\delta f_{\rm cusp}$.
	Therefore  the image of $\natural_p^n H_\delta$ on $\LtwoH$ is cuspidal.
	
Assume that $f$ is a self-dual Hecke-Maass form for $SL(n, \Z)$. 
By definition, we get $\ell_\infty(f) = -\ell_\infty(f)$ and $\ell_p(f) = -\ell_p(f)$ for any prime $p$, up to permutation. 
So $\widehat\natural_p^n(\ell_\infty(f), \ell_p(f))=0$ and $\natural_p^n f=0$ for any prime $p$. Therefore the image of $\natural_p^n H_\delta$ on $\LtwoH$ is generated by non self-dual Hecke-Maass forms. 
	
We already show that the image of $\natural_p^n H_\delta$ on $\LtwoH$ is non-trivial in  Lemma \ref{l:annihilate_image_nonzero}. 
So it remains to prove that 
the image is infinite dimensional. 

Take $\alpha_\infty=(\alpha_{\infty,1}, \ldots, \alpha_{\infty, n})$, $\alpha_p=(\alpha_{p, 1}, \ldots, \alpha_{p, n})\in \mathfrak a_\C^*(n)$ 
 such that $\widehat\natural_p^n(\alpha_\infty, \alpha_p)\neq 0$. Assume that $\rRe(\alpha_{v, j})=0$ for $1\leq j\leq n$, for $v=\infty$, $p$. 
 Take $0<\delta\leq \ln\left(\frac{n(n+6)}{8}\left(\sum_{j=1}^n \left|\ell_{\pi_\infty, j}+\frac{n-2j+1}{2}\right|\right)^{-1}+1\right)$, then we have 
  $\left|\widehat H_\delta(\alpha_\infty)\right|> \frac{1}{2}$
 by Lemma \ref{l:lowerbd_bump_funct}.

 As in Definition \ref{d:Quasi-Maass}, we construct a quasi-Maass form $F$ for $\{\alpha_\infty, \alpha_p\}$ such that 
	\begin{align*}
	F(z) &= \sum_{\gamma\in N(n-1, \Z)\bsl SL(n-1, \Z)}\sum_{\epsilon=\pm 1}\sum_{k_1, \ldots, k_{n-1}\geq 0} \frac{A_F\left(p^{k_1}, \ldots, p^{k_{n-1}}\right)}{p^{\frac{1}{2}\sum_{j=1}^{n-1} k_j(n-j)j}}\\
	&\hskip 50pt \times W_J\left(\bpm p^{k_1+\cdots +k_{n-1}} & & & \\ & p^{k_1+\cdots +k_{n-2}} & &  \\ & & \ddots & \\ & & & 1\ebpm \bpm \gamma & \\ & 1\ebpm \;z; \alpha_\infty, \epsilon\right)
	\end{align*}
where $A_F(p^{k_1}, \ldots, p^{k_{n-1}}) = S_{k_1, \ldots, k_{n-1}}(p^{-\alpha_{p, 1}},\ldots, p^{-\alpha_{p, n}})$.  
	Then 
	$$\natural_p^n H_\delta F(z) = \widehat\natural_p^n(\alpha_\infty, \alpha_p)\cdot\widehat H_\delta(\alpha_\infty)\cdot F(z)$$
for $z\in \bH^n$. 

Let $\widetilde F$ be the automorphic lifting of $F$ as Definition \ref{d:auto-lifting}. Then $\natural_p^n H_\delta\widetilde F \in \LtwoHcusp$ is smooth and cuspidal as we show above. 
By Lemma \ref{l:annihilate_image_nonzero},  $\natural_p^n H_\delta\widetilde F$ is non-trivial since  
 $\widehat H_\delta(\alpha_\infty)\cdot \widehat\natural_p^n(\alpha_\infty, \alpha_p)\neq 0$.

Assume that the space of the image of $\natural_p^n H_\delta$ on $\LtwoH$ is finite dimensional. 
Let 
	$$\natural_p^n \mathcal U:=\left\{u_j, \; \text{ a Hecke-Maass form of type }\mu_j\in \mathfrak a_\C^*(n)\;\left|\; \natural_p^n u_j\neq 0, \atop \text{ and }||u_j||_2^2=1\right.\right\}.$$
This finite set $\natural_p^n\mathcal U$ is not empty and it is an orthonormal basis of the image of $\natural_p^n H_\delta$ on $\LtwoH$. 
Let $B$ be the number of elements of $\natural_p^n\mathcal U$. 
Then $\natural_p^n\mathcal U = \left\{u_1, \ldots, u_B\right\}$ 
and there exist $c_1, \ldots, c_B\in \C$ such that
	\begin{align}\label{e:finite_expansion_testfunct}
	\natural_p^n H_\delta\widetilde F(z) = \sum_{j=1}^B c_j u_j(z).
	\end{align}
Here $c_j\neq 0$ for at least one $j=1, \ldots, B$. Assume that $c_1\neq 0$. 

Let $T> 4(2^{n-1}\ln p +\delta)$. 
For any $z\in \Sigma_{e^T}$, we have
	$$\left(\mathcal C_n^{(i)}-\lambda_\infty^{(i)}(\alpha_\infty)\right)\natural_p^n H_\delta \widetilde F(z) =\sum_{j=1}^B c_j \left(\mathcal C_n^{(i)}-\lambda_\infty^{(i)}(\alpha_\infty)\right)u_j(z) = 0$$
for any $i=1, \ldots, n-1$. Since $B$ is a finite positive integer, it is possible to assume that $\alpha_\infty\neq \mu_j$ (up to permutations) for any $j=1, \ldots, B$. 
Then there exists $i=1, \ldots, n-1$ such that
	$$\lambda_\infty^{(i)}(\mu_1)-\lambda_\infty^{(i)}(\alpha_\infty)\neq 0.$$
So there exists $c_2', \ldots, c_B'\in \C$ such that
	$$u_1(z) = \sum_{j=2}^B c_j'\cdot u_j(z)$$
for $z\in \Sigma_{e^T}$.
 So in a similar manner, we deduce that there exists $1\leq j\leq B$ such that 
	$$u_j(z) = 0$$
for $z\in \Sigma_{e^T}$. This gives a contradiction. 

Therefore $\natural_p^n \mathcal U$ should be an infinite set. It follows that the image of $\natural_p^n H_\delta$ on $\LtwoH$ is infinite dimensional.

\section{Proof}
\subsection{Proof of Theorem \ref{t:main}}
Let $M$ be a set of places over $\Q$ including $\infty$. For a given  set of local representations $\Pi_M$, we construct the quasi-Maass form $F_{\Pi_M}(z)$ for $\Pi_M$ and its automorphic lifting $\widetilde F_{\Pi_M}(z)$ with respect to the fixed fundamental domain $\fund^n$ as in Definition \ref{d:Quasi-Maass} and Definition \ref{d:auto-lifting} respectively. For each local representation $\pi_v\in \Pi_M$, we have the Satake (or Langlands) parameter $\ell_{\pi_v} = (\ell_{\pi_v, 1}, \ldots, \ell_{\pi_v, n})\in\mathfrak a_\C^*(n)$ as in (\ref{e:parameters_local_representation}). By Lemma \ref{l:lowerbd_bump_funct}, for a given $\delta$, the standard bump function $H_\delta$ satisfies $\left|\widehat H_\delta(\ell_{\pi_\infty})\right| > \frac{1}{2}$.
By Theorem \ref{t:image_cuspidal}, 
	$$\natural_p^nH_\delta \widetilde F_{\Pi_M}=\natural_p^n\left(\widetilde F_{\Pi_M}* H_\delta\right) \in \LtwoHcusp$$
	and $\natural_p^n H_\delta\widetilde F_{\Pi_M}$ is non-trivial since $\widehat\natural_p^n(\ell_{\pi_\infty}, \ell_{\pi_p})\cdot \widehat H_\delta(\ell_{\pi_\infty})\neq 0$. 

The key idea to prove Theorem \ref{t:main} is applying the following lemma to the cuspidal function $\natural_p^n H_\delta \widetilde F_{\Pi_M}(z)$. This lemma is a generalization of Lemma 1 \cite{Booker-Str-Venkatesh}. 

\begin{lem}\label{l:strategy}
	Let $n\geq 2$ be an integer. Let $M$ be a set of places of $\Q$ including $\infty$ and $\Pi_M$ be a set of local representations as in the introduction. 
	Let $S\subset M$ be a finite subset including $\infty$. 
	If there exists a non-zero smooth function $f\in\LtwoH$, which is cuspidal, such that 
	\begin{align}\label{e:strategy}
	\sum_{j=1}^{n-1}||\left(\mathcal C_n^{(j)}-\lambda_\infty^{(j)}(\ell_{\pi_\infty})\right)f||_2^2 +\sum_{q\in S, \atop \text{ finite }}\sum_{j=1}^{\lfloor\frac{n}{2}\rfloor}||\left(T_q^{(j)}-\lambda_q^{(j)}(\ell_{\pi_q})\right)f||_2^2 \; < \; \epsilon \cdot ||f||_2^2
	\end{align}
for $\pi_v\in \Pi_M$, for some $\epsilon>0$, 
then there exists an unramified cuspidal automorphic representation $\sigma$ of $\A^\times \bsl GL(n, \A)$ such that $d_S(\sigma,\Pi_M)<\epsilon$. 
\end{lem}

\begin{proof}[Proof of Lemma \ref{l:strategy}]
	By the spectral decomposition, the space $\LtwoHcusp$ is spanned by Hecke-Maass forms $u_i(z)$ with $||u_i||_2^2=1$ for $i=1, 2,\ldots$. For each $u_i$, there exists an unramified cuspidal automorphic representation $\sigma_i$ of $\A^\times\bsl GL(n, \A)$ such that $u_j$ is the Hecke-Maass form for $\sigma_j$. Then, for any $f\in \LtwoHcusp$, we have
		$$f(z) = \sum_{i=1}^\infty \left<f, u_i\right>\cdot u_i(z)\;.$$

For $\epsilon>0$, let
	$$\mathcal U_\epsilon(\Pi_M) :=\left\{u_i\;\left|\;d_S(\sigma_i, \Pi_M)<\epsilon\right.\right\}$$
and define
	$${\rm Pr}_\epsilon f(z):= \sum_{u_i\in\mathcal U_\epsilon(\Pi_M)} \left<f, u_i\right>\cdot u_i(z)\hskip 5pt\in \hskip 5pt \LtwoHcusp\;.$$

Assume that $f$ is smooth and satisfies (\ref{e:strategy}). Then 
	\begin{align*}
	&||{\rm Pr}_\epsilon f||_2^2 = ||f||_2^2 -\sum_{u_i\notin \mathcal U_\epsilon(\Pi_M)} \left|\left<f, u_i\right>\right|^2 \\
	 &\quad \geq ||f||_2^2\\
	&\quad - \sum_{i=1}^\infty \left|\left<f, u_i\right>\right|^2\cdot\frac{1}{\epsilon}\cdot\left\{\sum_{j=1}^{n-1}\left|\lambda_\infty^{(j)}(\sigma_i)-\lambda_\infty^{(j)}(\ell_{\pi_\infty})\right|^2 +\sum_{q\in S, \atop \text{ finite}} \sum_{j=1}^{\lfloor\frac{n}{2}\rfloor}\left|\lambda_q^{(j)}(\sigma_i)-\lambda_q^{(j)}(\ell_{\pi_q})\right|^2\right\}\\
	& \quad = ||f||_2^2-\frac{1}{\epsilon}\cdot \left\{\sum_{j=1}^{n-1}||\left(\mathcal C_n^{(j)}-\lambda_\infty^{(j)}(\ell_{\pi_\infty})\right) f||_2^2 +\sum_{q\in S, \atop \text{ finite}}\sum_{j=1}^{\lfloor\frac{n}{2}\rfloor}||\left(T_q^{(j)}-\lambda_q^{(j)}(\ell_{\pi_q})\right)f||_2^2\right\}\\
	&\quad > 0.
	\end{align*}
Therefore $\mathcal U_\epsilon(\Pi_M)\neq \emptyset$. 
\end{proof}

We are going to construct a formula for $\epsilon$ satisfying 
	\begin{align*}
	\frac{1}{||H_\delta \natural_p^n \widetilde F_{\Pi_M}||_2^2}\cdot &\left\{\sum_{j=1}^{n-1}||\left(\mathcal C_n^{(j)}-\lambda_\infty^{(j)}(\ell_{\pi_\infty})\right)H_\delta\natural_p^n \widetilde F_{\Pi_M}||_2^2\right.\\
	&\quad+\left.\sum_{q\in S, \atop \text{ finite}}\sum_{j=1}^{\lfloor\frac{n}{2}\rfloor}||\left(T_q^{(j)}-\lambda_q^{(j)}(\ell_{\pi_q})\right)H_\delta\natural_p^n\widetilde F_{\Pi_M}||_2^2\right\} <\epsilon.
	\end{align*} 
Then by Lemma \ref{l:strategy}, there exists an unramified cuspidal representation $\pi$ such that $d_S(\pi, \Pi_M)< \epsilon$.

The following lemma gives the lower bound for $||H_\delta\natural_p^n \widetilde F_{\Pi_M}||_2^2$.

\begin{lem}\label{l:lowerbd}
Let $n\geq 2$ be an integer. For a cuspidal function $f\in \LtwoH$, for $(m_1, \ldots, m_{n-1})\in \Z^n$, let
	\begin{multline*}
	W_f(z; m_1, \ldots, m_{n-2}, m_{n-1}) \\
	:=\int\limits_{(N(n, \R)\cap SL(n, \Z))\bsl N(n, \R)} f(uz)\; d^{-2\pi i(m_1u_{n-1, n}+\cdots +m_{n-2}u_{2, 3}+m_{n-1}u_{1, 2})}\; d^*u
	\end{multline*}
	where $u=\sm 1 & u_{1, 2}  & \ldots & u_{1, n}\\ & \ddots & & \vdots \\  && 1 & u_{n-1, n}\\ & & & 1\esm\in N(n, \R)$, for $z\in \bH^n$.
Then, for $T\geq 1$, we have
	$$||f||_2^2\; > \; \sum_{m_1=1}^\infty \cdots\sum_{m_{n-2}=1}^\infty \sum_{m_{n-1}\neq 0} \int_T^\infty\cdots\int_T^\infty \left|W_f(y; m_1, \ldots, m_{n-2}, m_{n-1})\right|^2\; d^*y$$
	where
	$y=a_{y_1, \ldots, y_{n-1}}\in A^0(n, \R)$.
\end{lem}

\begin{proof}[Proof of Lemma \ref{l:lowerbd}] 
We follow the argument in \S5.3, \cite{Goldfeld:2006}. For $j=1, \ldots, n-1$, let
	\begin{align*}
	u_{n-j+1} :=\bpm  & u_{1, n-j+1} &\\I_{n-j} & \vdots  & 0_{n-j\times j-1}\\ & u_{n-j, n-j+1}  &\\ & & \\ 0_{j\times n-j}& &I_j\\ & & \ebpm \in N(n, \R)
	\end{align*}
	where $u_{1, n-j+1}, \ldots, u_{n-j, n-j+1}\in \R$, $I_a$ is the $a\times a$ identity matrix and $0_{a\times b}$ is an $a\times b$ matrix with $0$ for every entry.
	
Let $z=xy\in \bH^n$ for $x=\sm 1 & x_{1, 2} & \ldots & x_{1, n}\\ & \ddots & & \vdots\\ & & 1 & x_{n-1, n}\\ & & & 1\esm\in N(n, \R)$ and $y=a_{y_1, \ldots, y_{n-1}}\in A_0(n, \R)$. Fix $j=1, \ldots, n-1$. For $m_1, \ldots, m_j\in \Z$, define
	\begin{align*}
	f_j(z;m_1, \ldots, m_j)& :=\int_{\Z\bsl \R}\cdots\int_{\Z\bsl \R} f\left(u_n\cdot u_{n-1}\cdots u_{n-j+1}\cdot z\right)\\\notag
	&\hskip 10 pt \times e^{-2\pi i(m_1u_{n-1, n}+\cdots +m_ju_{n-j, n-j+1})}\; d^*u_n\cdots d^*u_{n-j+1},
	\end{align*}
	where $d^*u_{n-j+1} = \prod_{k=1}^{n-j}du_{k, n-j+1}$. Then for $j=n-1$, we have
	$$f_{n-1}(z;m_1, \ldots, m_{n-1})=W_f(z;m_1, \ldots, m_{n-1}).$$
Let $f_0(z) := f(z)$ with $z\in \bH^n$. By the proof of Theorem 5.3.2, \cite{Goldfeld:2006},  we can prove the followings. 
\begin{itemize}
	\item For $j=1, \ldots, n-1$, we have
	\begin{align*}
	&f_j(z;m_1, \ldots, m_j)\\\notag
	&\hskip 20pt =\int_{\Z\bsl \R}\cdots\int_{\Z\bsl \R} f_{j-1}(u_{n-j+1}z; m_1, \ldots, m_{j-1})e^{-2\pi i m_ju_{n-j, n-j+1}}\; d^*u_{n-j+1}.
	\end{align*}
	
	\item Fix $j=1, \ldots, n-2$. For positive integers $m_1, \ldots, m_{j-1}$, we have
	\begin{align*}
	&f_{j-1}(z;m_1, \ldots, m_{j-1}) \\\notag
	&\hskip 10pt= \sum_{m_j=1}^\infty\sum_{\gamma_{n-j}\in P_{n-j-1, 1}(\Z)\bsl SL(n-j, \Z)}f_j\left(\bpm \gamma_{n-j} & \\ & I_j \ebpm z;m_1, \ldots, m_{j-1}, m_j\right).
	\end{align*}
	
	\item For positive integers $m_1,\ldots, m_{n-2}$, we have
	\begin{align*}
	f_{n-2}(z;m_1, \ldots, m_{n-2}) &=\sum_{0\neq m_{n-1}\in \Z}f_{n-1}(z;m_1, \ldots, m_{n-2}, m_{n-1})\\\notag
	&=\sum_{0\neq m_{n-1}\in \Z}W_f\left(z;m_1, \ldots, m_{n-2}, m_{n-1}\right).
	\end{align*}
\end{itemize}	

Since the Siegel set $\Sigma_{1, \frac{1}{2}}\subset \fund^n$, 
$$
	||f||_2^2 =\int_{\fund^n} \left|f(z)\right|^2\; d^*z\hskip 3pt>\hskip 3pt \int_1^\infty\cdots \int_1^\infty \int_{-\frac{1}{2}}^{\frac{1}{2}} \cdots \int_{-\frac{1}{2}}^{\frac{1}{2}} \left|f(z)\right|^2\; d^*z.
$$
Then
\begin{multline*}
	\int_1^\infty\cdots\int_1^\infty \int_{-\frac{1}{2}}^{\frac{1}{2}} \cdots \int_{-\frac{1}{2}}^{\frac{1}{2}} \left|f(z)\right|^2\; d^*z\\
	=\int_1^\infty\cdots\int_1^\infty \int_{-\frac{1}{2}}^{\frac{1}{2}}\cdots\int_{-\frac{1}{2}}^{\frac{1}{2}}\sum_{m_1=1}^\infty \sum_{\gamma_{n-1}\in P_{n-2, 1}(\Z)\bsl SL(n-1, \Z)} \overline{f(z)}e^{2\pi im_1(\gamma_{n-1, 1 }x_{1, n}+\cdots \gamma_{n-1, n-1}x_{n-1, n})}\\
	\times f_1\left(\bpm \gamma_{n-1} & \\ & 1\ebpm\cdot y_1^{-\frac{n-1}{n}}\cdot\bpm & &0\\ & y_1z' & \vdots\\ & & 0 \\ 0 & \ldots \; \; 0 & 1\ebpm ;m_1\right)\; d^*z, 
\end{multline*}
where $\gamma_{n-1} = \sm & & \\  & * & \\ & & \\ \gamma_{n-1, 1} & \ldots & \gamma_{n-1, n-1}\esm\in P_{n-2, 1}(\Z)\bsl SL(n-1, \Z)$. Here $z'\in \bH^{n-1}$. 
	For  a positive integer $m_1$ and $\gamma_{n-1} = \sm & & \\ & * &\\ & &  \\ \gamma_{n-1, 1} & \ldots & \gamma_{n-1, n-1}\esm\in P_{n-2, 1}(\Z)\bsl SL(n-1, \Z)$, it follows that
\begin{multline*} 
	\int_{-\frac{1}{2}}^{\frac{1}{2}}\cdots\int_{-\frac{1}{2}}^{\frac{1}{2}}\overline{f(z)}e^{2\pi im_1(\gamma_{n-1, 1 }x_{1, n}+\cdots \gamma_{n-1, n-1}x_{n-1, n})}\prod_{k=1}^{n-1}dx_k\\
	=\overline{f_1\left(\bpm \gamma_{n-1} & \\ & 1\ebpm \cdot y_1^{-\frac{n-1}{n}}\cdot \bpm & & 0\\ & y_1z' & \vdots \\& & 0\\ 0& \ldots \; \; 0&1\ebpm \right)}
\end{multline*}
So, 
\begin{multline*}
	\int_1^\infty\int_1^\infty\int_{-\frac{1}{2}}^{\frac{1}{2}}\cdots \int_{-\frac{1}{2}}^{\frac{1}{2}}\left|f(z)\right|^2\; d^*z\\
	\geq \sum_{m_1=1}^\infty\int_1^\infty\cdots\int_1^\infty \int_{-\frac{1}{2}}^{\frac{1}{2}} \cdots \int_{-\frac{1}{2}}^{\frac{1}{2}}\left|f_1\left(y_1^{-\frac{n-1}{n}}\cdot\bpm & &0\\ & y_1z' & \vdots\\ & & 0 \\ 0 & \ldots \; \; 0 & 1\ebpm; m_1\right)\right|^2\;  \prod_{1\leq i<j\leq n-1}dx_{i, j}\; d^*y.
\end{multline*}

After continuing this process inductively for $n-1$ steps, we finally obtain  
	\begin{align*}
	||f||_2^2\hskip 3pt > \hskip 3pt \sum_{m_1=1}^\infty\cdots \sum_{m_{n-2}=1}^\infty\sum_{m_{n-1}\neq 0}\int_1^\infty\cdots\int_1^\infty \left|W_f\left(y;m_1, \ldots, m_{n-2}, m_{n-1}\right)\right|^2\; d^*y\;.
	\end{align*}
 
\end{proof}

Let $T>  4(2^{n-1}\ln p +\delta)$. 
By Lemma \ref{l:annihilate_image_nonzero}, 
for any $z\in \Sigma_{e^T, \frac{1}{2}}\subset \fund^n$, we have
	$$H_\delta\natural_p^n \widetilde F_{\Pi_M}(z) = \widehat H_\delta(\ell_{\pi_\infty})\cdot \widehat\natural_p^n(\ell_{\pi_\infty}, \ell_{\pi_p})\cdot F_{\Pi_M}(z)\;.$$
Then for $z\in \Sigma_{e^T, \frac{1}{2}}$, we get
	\begin{multline*}
	W_{\natural_p^n H_\delta \widetilde F_{\Pi_M}}(z; 1, \ldots, 1) \\
	= \widehat H_\delta(\ell_{\pi_\infty}) \cdot \widehat\natural_p^n(\ell_{\pi_\infty}, \ell_{\pi_p})\cdot W_J(y; \ell_{\pi_\infty}, 1)e^{2\pi i(x_{n-2, n-1}+x_{n-2, n-1}+\cdots +x_{1, 2})}.
	\end{multline*}
Therefore by Lemma \ref{l:lowerbd}, we have
	$$
	||H_\delta\natural_p^n \widetilde F_{\Pi_M}||_2^2 > \frac{1}{4} \left|\widehat\natural_p^n(\ell_{\pi_\infty}, \ell_{\pi_p})\right|^2 \cdot\int_{p^T}^\infty\cdots\int_{p^T}^\infty \left|W_J(y;\ell_{\pi_\infty}, 1)\right|^2\; d^*y
	$$
for $y=a_{y_1, \ldots, y_{n-1}}\in A^0(n, \R)$. 
	
\begin{lem}\label{l:annihilating_norm}
Let $n\geq 2$ be an integer and $p$ be a prime. Then 
	\begin{align}\label{e:annihilating_norm}
	||\natural_p^n f ||_2^2
	\leq \left(p^{- \frac{n^2-1}{2(n^2+1)}} + p^{ \frac{n^2-1}{2(n^2+1)}}\right)^{n2^{n-1}}
	\cdot ||f||_2^2
	\end{align}
for any smooth function $f\in L^2\left(SL(n, \Z)\bsl\bH^n\right)$.  
\end{lem}

\begin{proof}
Since $\natural_p^n$ annihilates the continuous part, we should focus on the cuspidal functions. 
For any smooth cuspidal function $f\in L^2\left(SL(n, \Z)\bsl \bH^n\right)$, 
we have $f(z) = \sum_{j=1}^\infty \left<f, u_j\right>u_j(z)$, 
where $u_j(z)$ are Hecke-Maass forms for $SL(n, \Z)$ and $||u_j||_2^2=1$. So
	$$
	||\natural_p^n f||_2^2
	\leq \sup_{j}\left|\widehat\natural_p^n(\ell_\infty(u_j), \ell_p(u_j))\right|^2\cdot \|f\|_2^2.
	$$

For each $j\geq 1$, let $\ell_1:= \ell_{\infty}(u_j)$ and $\ell_2:=\ell_p (u_j)$. 
By \cite{Luo-Rudnick-Sarnak} and \cite{Luo-Rudnick-Sarnak:1999}, 
we have $\left|{\rm Re}(\ell_{1, i})\right|,  \left|{\rm Re}(\ell_{2, i})\right|\leq \frac{1}{2}-\frac{1}{n^2+1}$
for $i=1, \ldots, n$. 
Then for any $1\leq j_1< \cdots < j_k$ and $1\leq i_1< \cdots < i_k$ for $1\leq k\leq \lfloor\frac{n}{2}\rfloor$, we have
	$$\left|p^{-\frac{\ell_{1, i_1}+\cdots +\ell_{1, i_k}+\ell_{2, j_1}+\cdots +\ell_{2, j_k}}{2}} - p^{\frac{\ell_{1, i_1}+\cdots +\ell_{1, i_k}+\ell_{2, j_1}+\cdots +\ell_{2, j_k}}{2}}\right| \leq p^{-k\cdot \frac{n^2-1}{2(n^2+1)}} + p^{k\cdot \frac{n^2-1}{2(n^2+1)}}.$$
So, by (\ref{e: annihilating_eigenvalue}), we have 
	$$\left|\natural_p^n(\ell_1, \ell_2)\right| 
	\leq \prod_{k=1}^{\lfloor\frac{n}{2}\rfloor}\left(p^{- \frac{n^2-1}{2(n^2+1)}} + p^{ \frac{n^2-1}{2(n^2+1)}}\right)^{2k\cdot \frac{n!}{k!(n-k)!}}
	\leq \left(p^{- \frac{n^2-1}{2(n^2+1)}} + p^{ \frac{n^2-1}{2(n^2+1)}}\right)^{n2^{n-1}}.
	$$	
\end{proof}

\begin{lem}\label{l:upperbd_bump_funct}
Let $n\geq 2$ be an integer. For $\delta>0$ let $H_\delta$
be the standard bump function. For any $\ell = (\ell_1, \ldots, \ell_n)\in\mathfrak a_\C^*(n)$ with $\left|\rRe(\ell_j)\right|< \frac{1}{2}$, we have
	$$\left|\widehat H_\delta(\ell)\right| <  e^{\frac{n(n+6)}{4}\delta}.$$
\end{lem}
\begin{proof}[Proof of Lemma \ref{l:upperbd_bump_funct}]
For $\ell\in \mathfrak a_\C^*(n)$, we have
	\begin{align*}
	\left|\widehat H_\delta(\ell)\right| &= \left|\int_{SL(n, \R)}H_\delta(g)\cdot\varphi_\ell(g)\; dg\right| 
		\leq \underset{g\in B_\delta}\sup\left|\varphi_\ell(g)\right|
	\end{align*}
since $H_\delta(g)\geq 0$ and $\int_{B_\delta}H_\delta(g)\; dg=1$. 
As in the proof of Lemma \ref{l:lowerbd_bump_funct}, we have
	$$
	\left|\varphi_\ell(g)\right| 
	<e^{\frac{n(n+6)}{4}\delta},
	$$
	for any $g\in B_\delta$.
\end{proof}


The following lemma finally gives (\ref{e:main}).

\begin{lem}\label{l:upperbd_step1} 
Let $A_\infty$, $A_{S, {\rm finite}}$, $B_1$ and $B_2$ be as in Theorem \ref{t:main}. Then we have
	\begin{align*}
	&\sum_{j=1}^{n-1}||\left(\mathcal C_n^{(j)}-\lambda_\infty^{(j)}(\ell_{\pi_\infty})\right)\natural_p^nH_\delta \widetilde F_{\Pi_M}||_2^2
	 + \sum_{q\in S, \atop \text{ finite}}\sum_{j=1}^{\lfloor\frac{n}{2}\rfloor} ||\left(T_q^{(j)}-\lambda_q^{(j)}(\ell_{\pi_q})\right)\natural_p^nH_\delta \widetilde F_{\Pi_M}||_2^2 \\
	&\quad < \underset{z\in B_2}{\sup}\left|\widetilde F_{\Pi_M}(z)-F_{\Pi_M}(z)\right|^2\\
	&\quad\quad \times   \left(p^{- \frac{n^2-1}{2(n^2+1)}} + p^{ \frac{n^2-1}{2(n^2+1)}}\right)^{n2^{n-1}}
	\cdot {\rm Vol}\left(B_1\right) \cdot(A_\infty+A_{S, {\rm finite}}).
	\end{align*}
\end{lem}

\begin{proof}[Proof of Lemma \ref{l:upperbd_step1}]
Since the operator $\natural_p^n$ commutes with the invariant differential operators $\mathcal C_n^{(j)}$ and Hecke operators $T_q^{(j)}$ for $j=1, \ldots, n-1$, we have
	\begin{align*}
	&\sum_{j=1}^{n-1}||\left(\mathcal C_n^{(j)}-\lambda_\infty^{(j)}(\ell_{\pi_\infty})\right)\natural_p^nH_\delta \widetilde F_{\Pi_M}||_2^2 + \sum_{q\in S, \atop \text{ finite}}\sum_{j=1}^{\lfloor\frac{n}{2}\rfloor} ||\left(T_q^{(j)}-\lambda_q^{(j)}(\ell_{\pi_q})\right)\natural_p^nH_\delta \widetilde F_{\Pi_M}||_2^2 \\
	&\quad\quad < \left(p^{- \frac{n^2-1}{2(n^2+1)}} + p^{ \frac{n^2-1}{2(n^2+1)}}\right)^{n2^{n-1}}\\
	&\times \left\{\sum_{j=1}^{n-1}||\left(\mathcal C_n^{(j)}-\lambda_\infty^{(j)}(\ell_{\pi_\infty})\right)\widetilde F_{\Pi_M}*H_\delta||_2^2 + \sum_{q\in S, \atop \text{ finite}}\sum_{j=1}^{\lfloor\frac{n}{2}\rfloor} ||\left(T_q^{(j)}-\lambda_q^{(j)}(\ell_{\pi_q})\right)\widetilde F_{\Pi_M}* H_\delta||_2^2\right\}
	\end{align*}
by Lemma \ref{l:annihilating_norm}. 

Consider the case when $v=\infty$. Since 
$\left(\mathcal C_n^{(j)}-\lambda_\infty^{(j)}(\ell_{\pi_\infty})\right)F_{\Pi_M}*H_\delta\equiv 0$,
 it follows that
	$$
	||\left(\mathcal C_n^{(j)}-\lambda_\infty^{(j)}(\ell_{\pi_\infty})\right)\widetilde F_{\Pi_M}*H_\delta||_2^2 
	=||\left(\widetilde F_{\Pi_M}-F_{\Pi_M}\right)*\left(\mathcal C_n^{(j)}-\lambda_\infty^{(j)}(\ell_{\pi_\infty})\right)H_\delta||_2^2.
	$$
So, 
	\begin{align*}
	&||\left(\widetilde F_{\Pi_M}-F_{\Pi_M}\right)* \left(\left(\mathcal C_n^{(j)}-\lambda_\infty^{(j)}(\ell_{\pi_\infty})\right) H_\delta\right)||_2^2 \\
	&\leq \int_{\fund^n}\underset{g\in B_\delta}{\rm sup}\left|\left(\widetilde F_{\Pi_M}-F_{\Pi_M}\right)(zg^{-1})\right|^2\cdot \left(\int_{SL(n, \R)}\left|\left(\mathcal C_n^{(j)}-\lambda_\infty^{(j)}(\ell_{\pi_\infty})\right)H(g)\right|\; dg\right)^2\; d^*z\\
	&\leq \underset{\fund^n\cdot B_\delta-\fund^n}{\rm sup}\left|\widetilde F_{\Pi_M}-F_{\Pi_M}\right|^2 \cdot {\rm Vol}\left(\left(\bH^n-\widetilde \fund^n\right)\cdot B_\delta\cap \fund^n\right)\\
	&\quad\times \left(\int_{SL(n, \R)}\left|\left(\mathcal C_n^{(j)}-\lambda_\infty^{(j)}(\ell_{\pi_\infty})\right)H_\delta(g)\right|\; dg\right)^2.
	\end{align*}

Consider the case when $v=q<\infty$ and $q\in S$. Since the Hecke operators commute with the convolution operator, we have
	\begin{align*}
	&||\left(T_q^{(j)}-\lambda_q^{(j)}(\ell_{\pi_q})\right)\left(\widetilde F_{\Pi_M}* H_\delta\right)||_2^2 \; \leq \;  e^{\frac{n(n+6)}{4}\delta}\cdot ||\left(T_q^{(j)}-\lambda_q^{(j)}(\ell_{\pi_q})\right)\widetilde F_{\Pi_M}||_2^2 
	\end{align*}
by Lemma \ref{l:upperbd_bump_funct}. Since $T_q^{(j)}F_{\Pi_M} = \lambda_q^{(j)}(\ell_{\pi_q})\cdot F_{\Pi_M}$ and $\widetilde F_{\Pi_M}(z) - F_{\Pi_M}(z) = 0$ for $z\in \widetilde\fund^n$, we have
	\begin{align*}
	&||\left(T_q^{(j)}-\lambda_q^{(j)}(\ell_{\pi_q})\right)\widetilde F_{\Pi_M}||_2^2 =\int_{\fund^n}\left|T_q^{(j)}\left(\widetilde F_{\Pi_M}-F_{\Pi_M}\right)(z)\right|^2\; d^*z\\
	&\quad \leq\; \underset{T_q^{(j)}\fund^n -\fund^n}{\sup}\left|\widetilde F_{\Pi_M}-F_{\Pi_M}\right|^2 \cdot {\rm Vol}\left((T_q^{(j)})^{-1}\left(\bH^n-\widetilde \fund^n\right)\cap \fund^n\right)\cdot\left(\sharp T_q^{(j)}\right)^2.
	\end{align*}	
\end{proof}

\subsection{Proof of Theorem \ref{t:main_laplacian}}
For $\delta>0$, let $C_\delta$ be as in Theorem \ref{t:main_laplacian} and
	\begin{align}\label{e:bump_function_example}
	H_\delta(g) := \left\{\begin{array}{lll}
	C_\delta \cdot e^{-\frac{1}{1-\left(\delta^{-1}\sigma(g)\right)^2}}, & \text{ if }\sigma(g)<\delta\\
	0, & \text{ otherwise}
	\end{array}\right.\end{align}
for $g\in SL(n, \R)$. Then $H_\delta$ is a standard bump function. 

For $j=1, \ldots, n$, let $D_j = D_{j, j}$.
For $g=\xi_1 \exp(a)\xi_2$ for $\xi_1, \xi_2\in SO(2, \R)$ and $a=(a_1, \ldots, a_n)\in \mathfrak a(n)$ with $a_1> \cdots > a_n$, by Theorem 4.1, VII, \S4 in \cite{Jorgenson-Lang}, we have
	\begin{multline*}
	\Delta_n H_\delta(g) = \Delta_n H_\delta(\exp a) \\
	=-\frac{1}{n(n-1)}\left\{\sum_{j=1}^{n-1}\frac{1}{j^2+j}\left(\sum_{i=1}^j (D_i -D_{j+1})\right)^2 \right.\\
	\left. +\sum_{1\leq i< j\leq n}\coth(a_i-a_j)\left(D_{i}-D_j\right)\right\}H_\delta(\exp a).
	\end{multline*}
So we have
	$$\left|\Delta_n H_\delta(\exp a)\right| \leq \frac{3(1+e^{2\delta})}{\delta^4} C_\delta$$
for $g\in B_\delta$. 

Therefore 
	\begin{align}\label{e:laplacian_bump_function}
	\notag
	&\int_{SL(n, \R)} \left|\left(\Delta_n -\lambda_\infty(\pi_\infty)\right)H_\delta(g)\right|\; dg  \\\notag
	&\quad \leq \int_{SL(n, \R)}\left|\Delta_n H_\delta(g)\right|\; dg+ \int_{SL(n, \R)} \left|\lambda_n(\ell_{\pi_\infty}) H_\delta(g)\right|\; dg\\
	& \quad\leq \frac{3(1+e^{2\delta})}{\delta^4} C_\delta\cdot{\rm Vol}(B_\delta)+\left|\lambda_n(\ell_{\pi_\infty})\right|\;.
	\end{align}
By applying (\ref{e:laplacian_bump_function}) to Theorem \ref{t:main}, we prove Theorem \ref{t:main_laplacian}.


\thispagestyle{empty}
{\footnotesize
\nocite{*}
\bibliographystyle{amsplain}
\bibliography{ACT_reference}

\providecommand{\bysame}{\leavevmode\hbox to3em{\hrulefill}\thinspace}
\providecommand{\MR}{\relax\ifhmode\unskip\space\fi MR }
\providecommand{\MRhref}[2]{%
  \href{http://www.ams.org/mathscinet-getitem?mr=#1}{#2}
}
\providecommand{\href}[2]{#2}
\begin{thebibliography}{10}

\bibitem{Bian}
Ce~Bian, \emph{Computing ${GL}(3)$ automorphic forms}, Bull. London Math. Soc.
  (2010) \textbf{42(5)} (2010), 827--842.

\bibitem{Booker}
Andrew Booker, \emph{Uncovering a new {L}-function}, Notices of the {AMS}
  \textbf{55} (2008), 1088--1094.

\bibitem{Booker-Str-Venkatesh}
Andrew Booker, Andreas Str{\"o}mbergsson, and Akshay Venkatesh, \emph{Effective
  {C}omputation of {M}aass cusp forms}, IMRN (2006).

\bibitem{Brumley}
Farrell Brumley, \emph{Effective multiplicity on on ${GL}_{N}$ and narrow
  zero-free regions for {R}ankin-{S}elberg ${L}$-functions}, American Journal
  of Mathematics \textbf{128} (2006), 1455--1474.

\bibitem{Flath}
D.~Flath, \emph{Decomposition of representations into tensor products}, Proc.
  Sympos. Pure Math., XXXIII, Amer. Math. Soc., Providence, R.I. (1979),
  170--183.

\bibitem{Goldfeld:2006}
Dorian Goldfeld, \emph{Automorphic forms and {L}-functions for the group
  ${GL}(n, \mathbb{R})$}, no.~99, Cambridge studies in advanced mathematics,
  2006.

\bibitem{GH:2011}
Dorian Goldfeld and Joseph Hundley, \emph{Automorphic representations and
  ${L}$-functions for the general linear group}, no. 129, 130, Cambridge
  studies in advanced mathematics, 2011.

\bibitem{Grenier}
Douglas Grenier, \emph{On the shape of fundamental domains in ${GL}(n,
  \mathbb{R})$}, Pacific Journal of Math. \textbf{160} (1993), 53--65.

\bibitem{Hejhal}
Dennis~A. Hejhal, \emph{On eigenfunctions of the {L}aplacian for {H}ecke
  triangle groups. {I}n {E}merging applications of number theory
  ({M}inneapolis, {MN}, 1996)}, IMA Vol. Math. Appl. \textbf{109} (1999),
  291--315.

\bibitem{Helgasson:2000}
Sigurdur Helgasson, \emph{Groups and geometric analysis: Integral geometry,
  invariant differential operators, and spherical functions}, no.~83, American
  Mathematical Society, 2000.

\bibitem{Jacquet}
H.~Jacquet, \emph{{Functions de Whittaker associates aux groupes de Chevalley
  (French)}}, Bull.Soc. Math. France \textbf{95} (1967), 243--309.

\bibitem{Jorgenson-Lang}
Jay Jorgenson and Serge Lang, \emph{Spherical {I}nversion on
  ${SL}_n(\mathbb{R})$}, Springer {M}onographs in {M}athematics,
  Springer-Verlag New York, Inc., 2001.

\bibitem{Lindenstrauss-Venkatesh}
Elon Lindenstrauss and Akshay Venkatesh, \emph{Existence and {W}eyl's law for
  {S}pherical {C}usp forms}, GAFA \textbf{17} (2007), 220--251.

\bibitem{Luo-Rudnick-Sarnak}
W.~Luo, Z.~Rudnick, and P.~Sarnak, \emph{{On Selberg's Eigenvalue Conjecture}},
  GAFA \textbf{5} (1995), no.~2, 387--401.

\bibitem{Luo-Rudnick-Sarnak:1999}
\bysame, \emph{{On the generalized Ramanujan conjecture for $GL(n)$}}, Proc.
  Sym. Pure Math. \textbf{II} (1999), no.~66, 301--311.

\bibitem{Maass}
H.~Maass, \emph{{\"U}ber eine neue {A}rt von nichatanalytischen automorphen
  funktionen und die {B}estimmung {D}irichletscher {R}eihen durch
  {F}unktionalgleichungen}, Ann. Math. \textbf{122} (1949), 141--83.

\bibitem{Mezhericher}
Boris Mezhericher, \emph{Evaluating {J}acquet's {W}hittaker functions and
  {M}aass forms for ${SL}(3, \mathbb{Z})$}, {P}h.{D.} thesis, {C}olumbia
  {U}niversity (2008).

\bibitem{Miller}
Stephen~D. Miller, \emph{On the existence and temperedness of cusp forms for
  ${SL}_3(\mathbb{Z})$}, J. Reine Angew. Math. \textbf{533} (2001), 127--169.

\bibitem{Muller}
Werner M{\"u}ller, \emph{Weyl's law for the cuspidal spectrum of ${SL}_n$},
  Ann. of Math. \textbf{165} (2007), 275--333.

\bibitem{Selberg}
Atle Selberg, \emph{Harmonic analysis and discontinuous groups in weakly
  symmetric {R}iemannian spaces with applications to {D}irichlet series}, J.
  Indian Math. Soc. (N.S.) \textbf{20} (1956), 47--87.

\bibitem{Stark}
Harold Stark, \emph{Fourier coefficients of {M}aass waveforms}, Modular forms
  (Durham, 1983), Ellis Horwood Ser. Math. Appl.: Statist. Oper. Res. Horwood,
  Chichester (1984), 263--169.

\end{thebibliography}
}

\end{document}